\documentclass[12pt]{article}
\nonstopmode
\RequirePackage[colorlinks,citecolor=blue,urlcolor=blue,linkcolor=blue]{hyperref}
\hypersetup{
colorlinks = true,
citecolor=blue,
urlcolor=blue,
linkcolor=blue,
pdfauthor = {A. Kuznetsov, X. Peng},
pdfkeywords = {60G51, Levy process, Wiener-Hopf factorization, entire functions of Cartwright class, distribution of supremum, 
spectrally negative processes, scale function},
pdftitle = {On the Wiener-Hopf factorization for Levy processes with bounded positive jumps},
pdfpagemode = UseNone
}
\usepackage{graphicx,xspace,colortbl}
\usepackage{amsmath,amsthm,amsfonts}
\usepackage{color}
\usepackage{fancybox}
\usepackage{subfig}
\usepackage{pdfsync}
\usepackage{appendix}
    \def\qed{\hfill$\sqcap\kern-8.0pt\hbox{$\sqcup$}$\\}
    \def\beq{\begin{eqnarray}}
    \def\eeq{\end{eqnarray}}
    \def\beqq{\begin{eqnarray*}}
    \def\eeqq{\end{eqnarray*}}

    \def\re{\textnormal {Re}}
    \def\im{\textnormal {Im}}
    \def\arg{\textnormal {arg}}
    \def\p{{\mathbb P}}
    \def\e{{\mathbb E}}
    \def\r{{\mathbb R}}
    \def\c{{\mathbb C}}
    	
    \def\d{{\textnormal d}}
    \def\i{{\textnormal i}}
    
    \def\plus{{\scriptscriptstyle +}}
    \def\minus{{\scriptscriptstyle -}}
    
    \def\phiqp{\phi_q^{\plus}}
    \def\phiqm{\phi_q^{\minus}}

    \def\ee{{\textnormal e}}

    \def\mod{\textnormal {mod}}

    \def\xsup{{S}} 
    \def\xinf{{I}} 
    \def\pibar{\bar \Pi}

\newtheorem{theorem}{Theorem}
\newtheorem{lemma}{Lemma}
\newtheorem{proposition}{Proposition}

\theoremstyle{definition}
\newtheorem{definition}{Definition}
\newtheorem{example}{Example}

\newtheorem{remark}{Remark}

\title{
\textbf{On the Wiener-Hopf factorization for L\'evy processes with bounded positive jumps}
}
\author{
\textbf{
A. Kuznetsov
\footnote{Department of Mathematics and Statistics, 
York University, 
4700 Keele Street, 
Toronto, Ontario, 
M3J 1P3, Canada. Email: kuznetsov@mathstat.yorku.ca. Research supported by the
Natural Sciences and Engineering Research Council of Canada. }
}
,\, 
\textbf{
X. Peng
\footnote{Department of Mathematics, Hong Kong University of Science and Technology,
Clear Water Bay, Kowloon, Hong Kong.
Email: maxhpeng@ust.hk. Research supported
by the University Grants Committee of HKSAR of China and the Department of
Mathematics of HKUST.
}
}
}
\date{\footnotesize Current version: \today}

\begin{document}
\maketitle
\begin{abstract}
We study the Wiener-Hopf factorization for L\'evy processes with bounded positive jumps and arbitrary negative jumps. Using 
the results from the theory of entire functions of Cartwright class we prove that the positive 
Wiener-Hopf factor can be expressed as an infinite product in terms of the solutions to the equation $\psi(z)=q$, 
where $\psi$ is the Laplace exponent of the process. Under some additional regularity assumptions
on the L\'evy measure we obtain an asymptotic expression for these solutions, which is important for numerical computations. 
In the case when the process is spectrally negative with bounded jumps, we
 derive a series representation for the scale function in terms of the solutions to the equation $\psi(z)=q$. 
 To illustrate possible applications we discuss the implementation of numerical algorithms and present the results of several numerical experiments.

\bigskip


\medskip

\end{abstract}

\section{Introduction}\label{sec_intro}

Assume that we want to study the way in which one-dimensional L\'evy process $X$ exits a half-line or a finite interval. For example, we might be interested in the first passage time across a barrier, the overshoot/undershoot at the first passage time, the last time that the process was closest to the barrier, the location of the process at this time, etc.  These questions are usually referred to as ``exit problems" in the literature, and they have stimulated a lot of research in recent years due to numerous applications in such diverse areas as 
actuarial mathematics, mathematical finance, queueing theory and optimal control.  

Let us denote the  supremum  $\xsup_t=\sup\{X_s: 0\le s \le t\}$ and infimum $\xinf_t=\inf\{X_s: 0\le s \le t\}$, and let  $\ee(q)$ be an exponentially distributed random variable with parameter $q>0$, which is independent of the process $X$.
It is an established fact that exit problems are closely related to the Wiener-Hopf factorization, which studies the
distribution of  $\xsup_{\ee(q)}$ and $\xinf_{\ee(q)}$. For example, if we know the positive Wiener-Hopf factor (which is defined as the Laplace transform of $\xsup_{\ee(q)}$), then through the Pecherskii-Rogozin identity \cite{Pecherskii} we know the joint Laplace transform of the first passage time
and the overshoot. The bad news is that for general L\'evy processes the Wiener-Hopf factors cannot be obtained in closed form, therefore the best that we can do is to try to find rich enough families of L\'evy processes with special analytical properties, for which we can say something 
useful about the distribution of  $\xsup_{\ee(q)}$ and $\xinf_{\ee(q)}$. 

Let us look at the existing examples of L\'evy  processes for which one can identify the Wiener-Hopf factors and the distribution of extrema. In the general case, when the process has jumps of both sides, this list includes processes with jumps having rational transform \cite{Mordecki,Pistorius} and recently introduced meromorphic processes \cite{KuzKypJC}. The first class includes processes with hyper-exponential \cite{Cai,JP,Kou} and phase-type jumps \cite{Asmussen2,Asmussen}, while meromorphic processes include Lamperti-stable processes \cite{CaCha, Caballero2008,  Chaumont2009, Patie2009}, 
hypergeometric processes \cite{Caballeroetal2009, KyPavS, KyPaRi}, $\beta$-processes \cite{Kuz-beta} and $\theta$-processes \cite{Kuz-theta}. 
In the simpler case when the process is spectrally negative (which means essentially that it has only negative jumps) it turns out that both of the same two 
classes provide analytically tractable formulas, however in this case there also 
exist other interesting families, such as the processes constructed in \cite{HuKy} (see also \cite{KuKyRi}). 

One may wonder what is so special about these particular processes, that makes it possible to find the Wiener-Hopf factorization explicitly? It turns out that in all cases the Laplace exponent, defined as $\psi(z)=\ln ( \e[\exp(z X_1)])$, has some analytical structure which allows to factorize it as a product of two functions, which are analytic in the left/right complex halfplane. It is not surprising that 
the analytic structure of $\psi(z)$ plays such an important role, as there is a close
connection between Wiener-Hopf factorization and the Riemann boundary value problems, see \cite{Fourati},  \cite{Kuz2010c} and the references therein.  
For example, if the process has hyper-exponential jumps \cite{Cai}, then $\psi(z)$ is a rational function and if $X$ is a meromorphic process then $\psi(z)$ is a meromorphic function of a very special type, in both cases these functions can be easily factorized as products of two functions. One can formulate a general ``meta-theorem": Wiener-Hopf factorization can be obtained explicitly if and only if $\psi(z)$ can be extended to a meromorphic function in the left or right complex halfplane.  This principle helps to explain why no one has yet produced an explicit Wiener-Hopf factorization for one of the processes which are widely used in mathematical finance, such as VG, CGMY/KoBoL or generalized tempered stable processes (see \cite{Cont} and the references therein for more information about these families of L\'evy processes). It turns out that in all these cases the Laplace exponent has a logarithmic or algebraic branch point in the  complex plane, and, therefore, cannot be extended meromorphically. At the same time, we can use this meta-theorem to produce a large class of processes for which there is some hope to have an explicit Wiener-Hopf factorization: if the process has bounded jumps then it follows quite easily from the L\'evy-Khintchine formula that the Laplace exponent $\psi(z)$ is an entire function, and it might be possible to factorize it as a product of two functions and obtain some
useful information about the Wiener-Hopf factorization.

In this paper we consider a more general class: L\'evy processes with bounded positive jumps. There are two main reasons, one theoretical
and one more practical, why we are interested
in studying this class of processes. First of all, one can see that this is a very large class. In a certain sense it is ``dense" in the class of all L\'evy processes: clearly, any L\'evy measure can be approximated arbitrarily close by
truncating it at a large positive number. Therefore studying the Wiener-Hopf factorization for this class will lead to a better understanding 
of related results for general L\'evy processes. The second reason is that there are several situations where processes with bounded positive or negative jumps would be natural candidates for modeling purposes. One important example is ruin problem for the insurance company which is protected by the reinsurance agreement. In this case the size of each claim is essentially capped at a fixed level, and the amount of the claim  above this level is being covered by the reinsurer. The value of the insurance company can be conveniently modeled by a spectrally negative L\'evy process with bounded jumps, and now we have an interesting problem of how to compute numerically such important quantities 
as the ruin probability, discounted penalty function, etc.

It is instructive to draw a parallel with the results of Lewis and Mordecki  \cite{Mordecki} on processes with positive jumps of rational transform (see also recent paper by Fourati \cite{Fourati} on double-sided exit problem for this class of processes). In their case the Laplace exponent of the ascending ladder process $\kappa(q,z)$ (see \cite{Kyprianou} for the definition of this object) is a rational function, with all singularities in the left half-plane $\re(z)<0$. In our case it will turn out that $\kappa(q,z)$ is an entire function of a very special type: it belongs to the so-called Cartwright class (see \cite{Levin1980} and the proof of Theorem \ref{thm_main}). This makes it possible to factor it as an infinite product and to identify the Wiener-Hopf factors.  There are also some similarities between the analytical structure for L\'evy processes with bounded positive jumps and meromorphic processes \cite{KuzKypJC}. In both cases the positive Wiener-Hopf factor is given as an infinite product involving the solutions to the equation $\psi(z)=q$ in the half-plane $\re(z)>0$.  
The major difference is that in the case of 
meromorphic processes the solutions to the equation $\psi(z)=q$ are all real and simple, while they are complex when the process has bounded positive jumps, and this fact makes the analytical theory more interesting and the computations somewhat more challenging.

The paper is organized as follows: in Section \ref{sec_results} we present our main results on the Wiener-Hopf factorization for processes with bounded positive jumps, and we obtain an expression for the Wiener-Hopf factors as an infinite product in terms of the solutions to $\psi(z)=q$. 
We also study the asymptotics of these solutions, which will turn out to be very important for applications and numerical computations. 
In  Section \ref{sec_scale_functions} we consider the spectrally negative case, and we 
obtain a series representation for the scale function $W^{(q)}(x)$.  A brief discussion 
of numerical methods and the results of several numerical experiments are presented in Section \ref{sec_numerics}, 
while Section \ref{sec_proofs} contains the proofs of all results.

\section{Processes with bounded positive jumps}\label{sec_results}

Let us first introduce some notations and definitions. The L\'evy measure of the process $X$ will be denoted by $\Pi(\d x)$, and we will use 
the following notations for its tails:  $\pibar^+(x)=\Pi((x,\infty))$ for $x>0$ and $\pibar^-(x)=\Pi((-\infty,x))$ for $x<0$.  
In this paper we consider the class of processes with bounded positive jumps, 
thus we will assume that the L\'evy measure $\Pi$ has support on $(-\infty,k]$. Here $k$ is the right boundary of the support of $\Pi$, that is
 \beq\label{assmpt1}
  k=\inf\{ x>0: \pibar^+(x)=0\}.
 \eeq
We will also assume that $k>0$, so that we exclude the spectrally negative case, which will be considered in the next section. Note that at this 
stage we do not impose any restrictions on the L\'evy measure on the negative half-line.

The Laplace exponent of the process $X$ is defined as $\psi(\i z)=\ln(\e[\exp(\i zX_1)])$ for $z\in \r$, 
and it can be expressed by the  L\'evy-Khintchine formula as follows
\beq\label{Levy_Khinthine2}
\psi(z)=\frac12 \sigma^2 z^2 +\mu z + \int\limits_{-\infty}^{k} \left( e^{z x}-1- zx h(x) \right) \Pi(\d x), 
\eeq
where $\sigma \ge 0$, $\mu \in \r$ and $h(x)$ is the cutoff function. When the process has jumps of bounded variation, or equivalently, when
\beq\label{fin_var_condition}
\int\limits_{-1}^1 |x| \Pi(\d x) < \infty,
\eeq
we will assume that $h(x)\equiv 0$, then $\mu$ can be interpreted as the linear drift of the process. 
When the jump part of the process has infinite variation, or equivalently, when condition (\ref{fin_var_condition}) is violated, 
we will assume that $h(x)={\mathbf 1}_{\{x>-1\}}$ (or $h(x)\equiv 1$ if $\e[|X_1|]$ is finite). Note that formula (\ref{Levy_Khinthine2}) implies that the Laplace exponent $\psi(z)$ can be analytically continued into the half-plane $\re(z)>0$. Also note that $\psi(z)$ is real when $z>0$, 
and that $\overline{ \psi(z)}=\psi(\overline{z})$. In particular, the last property implies that if for some 
$q\in \r$ the number $z \in \c$ is a solution to
the equation $\psi(z)=q$, then so is $\overline{z}$.  

Everywhere in this paper we will denote the first quadrant of the complex plane as 
\beqq
{\mathcal Q}_1:=\{z\in \c: \; \re(z)>0, \; \im(z)>0\},
\eeqq
and we will always use the principal branch of the logarithm and the power function, that is the branch cut will be taken along the negative half-line and for all $z\in \c$ we have $\arg(z) \in (-\pi, \pi]$.

\subsection{Analytic properties of the Wiener-Hopf factors}\label{subsec_results}

The following theorem is our first main result. It describes the analytic structure of the Wiener-Hopf factors for processes with bounded positive
jumps. 
\begin{theorem}\label{thm_main}
Assume that $q>0$. Equation $\psi(z)=q$ has a unique positive solution $\zeta_0$ and 
infinitely many solutions in ${\mathcal Q}_1$, which we denote by $\{\zeta_n\}_{n\ge 1}$. Assume that $\zeta_n$ are arranged 
 in the order of increase of the absolute value.  
 The following statements are true:
 \begin{itemize}
  \item[(i)] $\zeta_0$ has multiplicity one and $\re(\zeta_n)\ge \zeta_0$ for all $n\ge 1$.
  \item[(ii)] The series  $\sum_{n\ge 1} \re\left(\zeta_n^{-1} \right)$ converges.
 \item[(iii)] All of the numbers $\{\zeta_n\}_{n\ge 1}$, except possibly those of a set of zero density, lie inside arbitrary small angle
 $\pi/2-\epsilon<\arg(z)<\pi/2$, and the density of zeros inside this angle is equal to
 \beq\label{density_of_zeros}
\lim\limits_{r\to +\infty} \frac{\#\{\zeta_n : |\zeta_n|<r \; \textnormal{ and } \; \pi/2-\epsilon<\arg(\zeta_n)<\pi/2\} }{r}=\frac{k}{2\pi}.
 \eeq
 \item[(iv)] The Wiener-Hopf factors can be identified as follows: for $\re(z)\ge 0$ 
 \beq\label{wh_factor}
\begin{cases}
\displaystyle \phiqp(\i z):=\e \left[ e^{- z S_{\ee(q)}} \right]= e^{\frac{ k z}2 } \left( 1+\frac{z}{\zeta_0}\right)^{-1}
  \prod\limits_{n\ge 1} \left(1+\frac{ z}{\zeta_n}\right)^{-1}\left(1+\frac{z}{\bar \zeta_n}\right)^{-1}, \\
 \displaystyle  \phiqm(-\i z):= \e \left[ e^{z I_{\ee(q)}} \right]=\frac{q}{q-\psi(z)} \frac{1}{\phiqp(-\i z)}.
\end{cases}
 \eeq
 \end{itemize}
\end{theorem}

The proof of Theorem \ref{thm_main} can be found in Section \ref{sec_proofs}. 

\vspace{0.5cm}

\label{page_examples}
Let us present a very simple example, which will illustrate the results presented in Theorem \ref{thm_main}. Consider a process $X_t=k N_t$, where $N_t$ is the standard Poisson process. It is clear that $X$ is a process with bounded
positive jumps, and that its Laplace exponent is $\psi(z)=\exp(kz)-1$. Solving equation $\psi(z)=q$ for $q>0$ we find that
\beqq
\zeta_0=\frac{\ln(1+q)}{k}, \;\;\; \zeta_n=\frac{\ln(1+q)}{k}+\frac{2n\pi \i}{k}, \;\;\; n\ge 1.
\eeqq
It is an easy exercise to verify that the series  $\sum_{n\ge 1} \re\left(\zeta_n^{-1} \right)$ converges, thus we have checked part (ii) of the Theorem \ref{thm_main}. Next, all the zeros belong
to the vertical line $\re(z)=\zeta_0$, they are
equidistant and the spacing between them is equal to $2\pi/k$. This confirms statement (iii): all zeros (except for a finite number) lie 
inside arbitrary small angle
 $\pi/2-\epsilon<\arg(z)<\pi/2$, and the density of zeros inside this angle, which is inversely proportional to the spacing, is equal to $k/(2\pi)$. 
 Finally, since $X$ is a subordinator, we have $S_{\ee(q)}\equiv X_{\ee(q)}$, thus 
 \beqq
  \e \left[ e^{- z S_{\ee(q)}} \right]=\e \left[ e^{- z X_{\ee(q)}} \right]=\frac{q}{q-\psi(-z)}=
  \frac{q}{1+q-e^{-kz}}=e^{\frac{kz}2} \frac{\sinh\left(\frac12 \ln(1+q)\right)}{\sinh\left(\frac12( kz+ \ln(1+q))\right)},
 \eeqq
 and we see that the infinite product representation in (\ref{wh_factor}) is equivalent to the well-known 
 infinite product formula for the hyperbolic sine function. 

It is also easy to verify the validity of Theorem \ref{thm_main} for a  more general class of processes with double-sided jumps. Let us assume that for some $h>0$ 
the measure $\Pi(\d x)$ is supported on a finite subset of a lattice $h {\mathbb Z}$, that is there exist  $m,l \in {\mathbb N}$ such that
the support of $\Pi(\d x)$ is equal to  $\{-mh, -(m-1)h,\dots,-h,h,\dots,(l-1)h,lh\}$. In this case the right boundary of the support of the L\'evy measure is $k=lh$. Let $X$ be a compound Poisson process defined by the measure
$\Pi(\d x)$ (note that $X$ can be constructed as a linear combination of $m+l$ independent Poisson processes). 
From the L\'evy-Khintchine formula \eqref{Levy_Khinthine2} we find that the Laplace exponent of $X$ is given by
\beqq
\psi(z)=\sum\limits_{j=1}^m \Pi(\{-jh\})\left( e^{-jh z}-1 \right)+\sum\limits_{j=1}^l \Pi(\{jh\})\left( e^{jh z}-1 \right).
\eeqq
Note that the function $\psi(\ln(w)/h)$ is a rational function, therefore using the change of variables $z=\ln(w)/h$ the equation $\psi(z)=q$ can be transformed into a polynomial equation of degree $m+l$. It is possible to prove (we leave
it as an exercise) that this polynomial equation will have exactly $m$ solutions inside the open unit circle $\{w\in \c: |w|<1\}$ and exactly
$l$ solutions $\{w_1,\dots,w_l\}$ outside the closed unit circle. The solutions $\zeta_n$ to the original  equation $\psi(z)=q$ can now be found by solving equation $\exp(hz)=w_j$, thus they are given by
\beqq
\left\{ \frac{\ln(w_j)}h+\frac{2n\pi \i}{h}; \; n\in {\mathbb Z}, \; 1\le j \le l\right\}.
\eeqq
Again, it is easy to check that the series $\sum_{n\ge 1} \re\left(\zeta_n^{-1} \right)$ converges. Also, the solutions lie on $l$ vertical lines, therefore all of them (except for a finite number) 
lie inside arbitrary small angle $\pi/2-\epsilon<\arg(z)<\pi/2$,  and the density of zeros inside this angle is equal to $l\times h/(2\pi)=k/(2\pi)$. Infinite product representation for the positive Wiener-Hopf factor (\ref{wh_factor}) is again equivalent to elementary infinite product expressions for certain trigonometric functions. 

In the above two examples we were able to describe the solutions to the equation $\psi(z)=q$ in a very precise form, but in the general case this will be a transcendental equation and there is little hope to 
obtain as much information about $\zeta_n$. 
However, as our next result shows, we can obtain some very useful information about 
the asymptotic behavior of $\zeta_n$, provided that the Laplace exponent $\psi(z)$ has regular growth as $z\to \infty$. The connection between
the regularity of growth and distribution of the zeros of an entire function is well-known, see, for example, chapters 2 and 3 in \cite{Levin1980}. 
The main idea is that  analytic functions which grow regularly at infinity also enjoy certain regularity in the distribution
of zeros (and in fact the opposite is also true). In the following Theorem we impose a rather strong regularity condition on the 
growth of the Laplace exponent in the half-plane $\re(z)>0$ in order to obtain an explicit asymptotic approximation for the solutions to 
$\psi(z)=q$. This asymptotic expression for $\zeta_n$ would prove to be very useful in the next Section, when we will derive a series representation for the scale function $W^{(q)}(x)$, and later in Section \ref{sec_numerics}, when we will discuss numerical algorithms.

\begin{theorem}\label{thm_asymptotics}
Assume that 
\beq\label{psi_asymptotics}
\psi(z)=Ae^{kz}z^{-a}+Bz^b+ o\left(e^{kz}z^{-a}\right) + o \left(z^b\right),
\eeq
as $z\to \infty$ in the domain ${\mathcal Q}_1$. Let us also assume that $a\ge 0$ and $b>0$. Then all sufficiently large
solutions to $\psi(z)=q$ are simple and there exists $m \in {\mathbb Z}$ such that
\beq\label{zeta_asymptotics}
\zeta_{n+m}&=&\frac{1}{k}  \left[\ln \left( \bigg | \frac{B}{A} \bigg| \right)+(a+b) \ln \left(\frac{2n\pi }{k}\right) \right] \\ \nonumber
&+&
\frac{\i }{k} \left[ \arg\left( \frac{B}{A} \right)+\left( \frac12 (a+b) + 2n+1\right) \pi \right]+o(1)
\eeq
as $n \to +\infty$. 
\end{theorem}

\noindent The proof of Theorem \ref{thm_asymptotics} is presented in Section \ref{sec_proofs}. 

\begin{remark}
Note that formula (\ref{zeta_asymptotics}) implies that $\zeta_n=(a+b)\ln(n)/k+2 \pi n \i /k+O(1)$ as $n\to \infty$, which
again confirms statements (ii) and (iii) of Theorem
\ref{thm_main}: the zeros cluster ``close'' to the imaginary axis, or to say it more precisely, $\arg(\zeta_n)\nearrow \pi/2$ as $n\to \infty$;
and secondly, the density of zeros in ${\mathcal Q}_1$ (which is inversely proportional to the average spacing between them) is equal to $k/(2\pi)$. 
\end{remark}

Condition $b>0$ in Theorem \ref{thm_asymptotics} implies that $\psi(\i z)\to \infty$ as $z\to \infty$, $z\in \r$, therefore $X$ cannot be a compound Poisson process
(see Proposition 2 in \cite{Bertoin}). This shows that the two examples considered on page \pageref{page_examples} do not satisfy the conditions 
of Theorem \ref{thm_asymptotics}, but if we take these compound Poisson processes and add a drift (or Brownian motion with drift) then it is easy to check that the Laplace exponent of this perturbed process 
will satisfy \eqref{psi_asymptotics}. A natural question  then is to describe sufficient conditions on the triple $\{\mu,\sigma,\Pi\}$, which
 defines the Laplace exponent via (\ref{Levy_Khinthine2}), which will ensure that $\psi(z)$ satisfies asymptotic relation (\ref{psi_asymptotics}). Below we present a set of sufficient conditions.

\begin{definition}
We will say that a real function $f(x)$ is piecewise $n$-times continuously differentiable on an interval $[a,b]$ if there exists a finite set of numbers $\{x_k\}_{1\le k \le m}$, such that
\begin{itemize}
 \item[(i)] $a=x_1<x_2<\dots<x_m=b$,
 \item[(ii)] $f \in {\mathcal C}^{n}([a,b] \setminus \{x_1,x_2,\dots,x_m\})$,
 \item[(iii)] for each $j=0,1,\dots,n$ and  $k=1,2,\dots,m$ there exist left and right limits $f^{(j)}(x_k-)$ and $f^{(j)}(x_k+)$.  
\end{itemize}
We will use the notation $f \in {\mathcal {PC}}^{n}[a,b]$. In the case of an open interval $(a,b)$ the definition of ${\mathcal PC}^{n}(a,b)$
is
very similar, except for condition $a<x_1$ and $x_m<b$. Similarly, one can define the remaining cases of intervals $(a,b]$ and $[a,b)$. 
\end{definition}

\begin{definition}\label{definition_Levy_measure} We say that a L\'evy measure is {\it regular} if the following two conditions are satisfied:
\begin{itemize}
 \item[(1)] There exist constants $\hat C, \hat \alpha$ and $\{\hat C_j,\hat \alpha_j\}_{1\le j \le \hat m}$  such that
\beq\label{pibar_minus_assmptn}
\pibar^-(x)-\hat C|x|^{-\hat \alpha}-\sum\limits_{j=1}^{\hat m} \hat C_j |x|^{-\hat \alpha_j} = O(1),\;\;\; x\to 0^-, 
\eeq
where  $\hat \alpha, \hat \alpha_j \in (-\infty,2)\setminus\{0,1\}$ and $\hat \alpha_j<\hat \alpha$.
\\

 \item[(2)] There exists $n \in {\mathbb N}\cup\{0\}$ such that 
 \begin{itemize}
 \item[(2a)] for some constants $C, \alpha$ and $\{C_j,\alpha_j\}_{1\le j \le m}$  we have
\beq\label{pibar_plus_assmptn}
\pibar^+(x)-Cx^{-\alpha}-\sum\limits_{j=1}^{m} C_j x^{-\alpha_j} \in {\mathcal {PC}}^{n+1}[0,k],
\eeq
where  $\alpha, \alpha_j \in (-\infty,2)\setminus\{0,1\}$ and $\alpha_j<\alpha$;
 \item[(2b)] $\pibar^+{}^{(n)}(k-) \ne 0$;
 \item[(2c)] $\pibar^+(x)\in {\mathcal C}^{n-1}(\r^+)$ (this condition is not needed for $n=0$).
 \end{itemize}
\end{itemize}
\end{definition}

\begin{remark}
Note that conditions (1) and (2a) imply that the Blumenthal-Getoor index 
\beq
\beta(\Pi)=\inf\left\{ \gamma>0 : \int_{-1}^1 |x|^{\gamma} \Pi (\d x)  < \infty \right \}
\eeq
is equal to $\beta(\Pi)=\max(\alpha,\hat \alpha,0)$.
\end{remark}

Definition \ref{definition_Levy_measure} is not very easy to interpret, therefore we will try to give some intuition behind these conditions.
 Conditions (1) and (2a) guarantee that the L\'evy measure is sufficiently well-behaved in the neighborhood of zero. This will help us to ensure 
 that the main term of $\psi(z)$ grows as $z\to \infty$ exactly as $z^b$, and does not contain any logarithmic terms.  
 Conditions (2b) and (2c) are slightly harder to interpret. Essentially, 
 they imply that the L\'evy measure restricted to $\r^+$ has its ``worst''
possible singularity at the right-end point of its support. Let us consider the following example, where conditions (2b) and (2c) 
are violated.   

 \begin{example}
Assume that the L\'evy measure is given by
\beq\label{example_Pi_dx}
\Pi(\d x)= {\mathbf 1}_{\{x<0\}}e^{x }\d x + {\mathbf 1}_{\{0<x<4\}} \d x + \delta_3(\d x),
\eeq
so that $\Pi(\d x)$ has an atom of mass one at $x=3$. 
Because of the atom at $x=3$ we know that $\pibar^+(x)$ is not continuous, therefore we are forced to take $n=0$ in the Definition \ref{definition_Levy_measure}. But since we have no atom at $x=k=4$, we find that $\pibar^+(k-)=0$, which violates condition (2b), thus we conclude that the measure $\Pi(\d x)$ is not regular. 

Next, let $X$ be a process which has a L\'evy measure \eqref{example_Pi_dx} and linear drift $\mu=1$. We will 
check that the Laplace exponent of the process $X$ does not satisfy \eqref{psi_asymptotics}. 
 We compute the Laplace exponent using the L\'evy-Khintchine formula 
\eqref{Levy_Khinthine2}
and find that it has asymptotics
\beqq
\psi(z)=z+\frac{e^{4z}}z+e^{3z}+O(1), 
\eeqq
as $z\to \infty$, $\re(z)>0$. Now, it is easy to see that in the domain $0<\re(z)<\frac12\ln|z|$ we will have
$\psi(z)=z+e^{3z}+o(e^{3z})$,
while in the domain $\re(z)>2\ln|z|$ we'll have $\psi(z)=e^{4z}/z+o(e^{4z}/z)$. This implies that we cannot find a single uniform asymptotic formula for $\psi(z)$
as in \eqref{psi_asymptotics}. This happens because the ``worst" singularity of $\Pi(\d x)$, which is the atom at $x=3$, is not located at the right boundary $x=k=4$. One can also check that asymptotic expression \eqref{psi_asymptotics} will be satisfied 
if we replace $\delta_3(\d x)$ by $\delta_4(\d x)$ in \eqref{example_Pi_dx} or if we add a second atom at $x=4$, 
and at the same time this will also give us a regular L\'evy measure according to the Definition \ref{definition_Levy_measure}. 
\end{example}

In the following example we exhibit a large family of regular L\'evy measures (it is an easy exercise to verify all 
the conditions of Definition \ref{definition_Levy_measure}). 
\begin{example}
Assume that the L\'evy measure  $\Pi(\d x)$ has a density $\pi(x)$ given by
\beq\label{pi_example}
\pi(x)=
  {\mathbf 1}_{\{x<0\}}\hat f(x)|x|^{-1-\hat \alpha}+ {\mathbf 1}_{\{0<x<k\}} f(x)x^{-1- \alpha} ,
\eeq
where 
$\alpha, \hat\alpha \in (-\infty,2)\setminus\{0,1\}$ 
and functions $f$, $\hat f$ satisfy the following conditions:(i) $f(x)$ and $\hat f(x)$ can be represented by convergent Taylor series in some neighborhood of zero;
(ii) $f(x)$ is ${\mathcal {PC}}^1[0,k]$; (iii) $f(k-)>0$. Then $\Pi(\d x)$ is regular. 
\end{example}

The above example shows that there are indeed many interesting L\'evy processes with regular L\'evy measure. For example, we can take one of the widely used processes in mathematical finance, such as CGMY/KoBoL or generalized tempered stable (see \cite{Cont}), truncate its L\'evy measure at any positive number and we will obtain a regular L\'evy measure.

The next Proposition shows that if the L\'evy process has a regular L\'evy measure, then its Laplace exponent 
satisfies the asymptotic expansion (\ref{psi_asymptotics}) and therefore the roots $\zeta_n$ have simple asymptotic approximation
given by \eqref{zeta_asymptotics}.
\begin{proposition}\label{prop_asymptotic_psi} 
 Assume that $X$ is not a compound Poisson process and that the L\'evy measure of $X$ is regular. Then
the asymptotic expression (\ref{psi_asymptotics}) is true, with parameters 
$ A=(-1)^n \pibar^+{}^{(n)}(k-)$, and $a=n$.
 The parameters $B$ and $b$ can be identified as follows:
\begin{itemize}
 \item[(i)] if $\sigma>0$, then $B=\sigma^2/2$ and $b=2$,
 \item[(ii)] if the process has paths of bounded variation and $\mu\ne 0$, then $B=\mu$  and $b=1$.
\end{itemize}
In the remaining cases, when the process has paths of unbounded variation and $\sigma=0$, or when the process has paths of bounded variation and $\mu=0$, we have
$b=\beta(\Pi)=\max(\alpha,\hat \alpha)$ and 
\beqq
B=
\begin{cases}
 -C e^{-\pi \i \alpha} \Gamma(1-\alpha), \; &\textnormal{ if } \; \alpha>\hat \alpha, \\
-\hat C \Gamma(1-\hat \alpha), \; &\textnormal{ if } \; \alpha<\hat \alpha,\\
-(C e^{-\pi \i \alpha} + \hat C)\Gamma(1-\alpha), \;  &\textnormal{ if } \; \alpha=\hat \alpha.
\end{cases}
\eeqq
Moreover, the asymptotic expression for $\psi'(z)$ can be obtained from (\ref{psi_asymptotics}) by taking derivative of the right-hand side.  
\end{proposition}

\noindent The proof of Proposition \ref{prop_asymptotic_psi} can be found in Section \ref{sec_proofs}.

\subsection{Partial fraction decomposition and distribution of $S_{\ee(q)}$}\label{subsection_conjecture}  

As we have mentioned in Section \ref{sec_intro}, in the case when $X$ has positive jumps of rational transform the 
positive Wiener-Hopf factor is a rational function (see \cite{Mordecki}). By performing the partial fraction decomposition
 of this rational function and inverting the Laplace transform one can obtain the distribution of $S_{\ee(q)}$ explicitly. 
 The same procedure works for meromorphic processes (see \cite{KuzKypJC}), though in this case we must work with meromorphic functions
 instead of rational functions, and things become slightly more technical. In our case, when the process has bounded positive jumps,
formula \eqref{wh_factor} tells us that the positive Wiener-Hopf factor $\phiqp(\i z)$ is also a meromorphic function, thus 
we might hope to follow the same procedure and obtain a series representation for the distribution of $S_{\ee(q)}$, which would be 
very useful for applications. Unfortunately it turns out that proving existence of the partial fraction decomposition 
for  $\phiqp(\i z)$ of the form \eqref{wh_factor} is a much harder problem, and we were not able to give a 
completely rigorous proof of such a result or to find such a result in the existing literature. 
However, if one assumes that such a partial fraction decomposition exists, it is rather easy to obtain the form of its coefficients and 
 the resulting expression for the distribution of $S_{\ee(q)}$. Let us sketch here the main steps, and 
 later, in Section \ref{sec_numerics} we perform several numerical experiments which seem to confirm our conjecture. 

\label{discussion_simple_zeros}

First of all, it is very likely that all the zeros  $\{\zeta_n\}_{n\ge 1}$ of the function $\psi(z)-q$ are simple. This 
was proved in Theorem \ref{thm_asymptotics} for large zeros and for $\zeta_0$. For other roots one could use the following (rather informal) argument. We know that $z$ is a solution of $\psi(z)=q$ with multiplicity greater than one if and only if $\psi'(z)=0$. We can rephrase this statement: equation $\psi(z)=q$ has a solution of multiplicity greater than one if only if $q=\psi(\zeta)$, where $\zeta$ is a root of $\psi'(z)$.  But $\psi'(z)$ is analytic in the half-plane
$\re(z)>0$, thus it has a discrete set of zeros, and for any given complex root of $\psi'(z)$ it is very unlikely that $\psi(\zeta)$ will be a real positive number, thus it is very unlikely that there exists $q>0$ such that $\psi(z)-q$ has a multiple root. Even if such a value of $q$ exists,
we see that in the worst possible case, there can be only finitely many such values on any compact subset of $\r$, this shows that the assumption that all the zeros $\{\zeta_n\}_{n\ge 1}$ are simple should not be very restrictive for practical purposes. 
However, of course this fact would require a rigorous proof in future work. 

Next, assuming that $0$ is regular for $(0,\infty)$, or equivalently, that the distribution of $S_{\ee(q)}$ has no atom at zero, 
we must have $\phiqp(\i z)=\e[\exp(-z S_{\ee(q)})]\to 0$ as $\re(z)\to +\infty$, thus 
it is reasonable to expect that the partial fraction decomposition for $\phiqp(\i z)$ should be of the form 
\beq\label{phiqp_partial_fractions}
  \phiqp(\i z)=\e\left[e^{-z S_{\ee(q)}}\right]= e^{\frac{ k z}2 } \left( 1+\frac{z}{\zeta_0}\right)^{-1}
   \prod\limits_{n\ge 1} \left(1+\frac{ z}{\zeta_n}\right)^{-1}\left(1+\frac{z}{\bar \zeta_n}\right)^{-1}=
   \frac{a_0}{z+\zeta_0}+ \sum\limits_{n\ge 1} \left[ \frac{a_n}{z+\zeta_n}+\frac{\bar a_n}{z+\bar \zeta_n} \right] 
\eeq
where $a_n={\textnormal{Res}}(\phiqp(\i z): \; z=-\zeta_n)$ for $n\ge 0$. 
Using the above infinite product representation we can easily compute the residues at points $\zeta_n$ and obtain   
\beqq
a_0=\zeta_0e^{-\frac12 k \zeta_0} \prod\limits_{m\ge 1} \bigg | 1-\frac{\zeta_0}{\zeta_m} \bigg |^{-2}, 
\eeqq
and
\beqq
a_n=\frac{\zeta_0 |\zeta_n|^2 e^{-\frac12 k \zeta_n}}{2\im(\zeta_n)(\zeta_n-\zeta_0)} 
 \prod\limits_{\substack{m\ge 1 \\ m\ne n}} \left[ \left(1-\frac{\zeta_n}{\zeta_m}\right) \left(1-\frac{\zeta_n}{\bar \zeta_m}\right) \right]^{-1}. 
\eeqq
Therefore, provided that there exists a partial fraction decomposition of the form (\ref{phiqp_partial_fractions}), we can use the uniqueness of Laplace transform and conclude that the density of the supremum $\xsup_{\ee(q)}$  should be given by the following infinite series 
 \beq\label{distribution_S_ee(q)}
 \frac{\d }{\d x} \p(\xsup_{\ee(q)}\le x) =  a_0 e^{- \zeta_0 x} + 
 2\sum\limits_{n\ge 1} \re\left[a_n e^{- \zeta_n x} \right].
 \eeq
Again, we emphasize that at this point the existence of a partial fraction decomposition (\ref{phiqp_partial_fractions}) is just a conjecture, which would have to be proven rigorously in future work.

\section{Scale functions for spectrally negative processes with bounded jumps}\label{sec_scale_functions}

Everywhere in this section we will assume that $Y$ is a spectrally negative L\'evy process with bounded jumps, 
so that the L\'evy measure $\Pi_Y(\d x)$ is 
supported on the interval $[-k,0)$ where $k>0$, and $k$ is the smallest such number. From the L\'evy-Khintchine formula (\ref{Levy_Khinthine2}) it follows that in this case 
the Laplace exponent $\psi_Y(z)=\ln \e [\exp(z Y_1)]$ is given by
\beqq
\psi_Y(z)=\frac12 \sigma^2 z^2 +\mu z + \int\limits_{-k}^{0} \left( e^{z x}-1- zx h(x) \right) \Pi_Y(\d x), 
\eeqq
and we see that $\psi_Y(z)$ is an entire function which is convex for real values of $z$. Since $\psi_Y(0)=0$, it is clear 
that for each $q>0$ the equation $\psi_Y(z)=q$ has a unique positive solution $z=\Phi(q)$, and in fact it is known from the general theory of 
spectrally negative processes that 
this is a unique solution in the half-plane $\re(z)>0$, see chapter 8 in \cite{Kyprianou}.

For $q>0$ the scale function $W^{(q)}(x)$ is  defined as follows: $W^{(q)}(x)=0$ for $x<0$ and on $[0,\infty)$ it is characterized 
via the Laplace transform identity
\beq\label{def_W^q}
\int\limits_{0}^{\infty} e^{-zx} W^{(q)}(x) \d x = \frac{1}{\psi_Y(z)-q}, \;\;\; \re(z) > \Phi(q).
\eeq 
The scale function can be considered as the main building block for the vast majority of fluctuation identities for spectrally negative processes, see 
\cite{Kyprianou,KuKyRi} for many examples of such identities. Here we will present one fundamental identity, which
is related to the exit of the process $Y$ from an interval. If we define the first passage time $\tau_a^+=\inf\{t>0: \; Y_t>a\}$ and similarly
$\tau_0^-=\inf\{t>0: \; Y_t<0\}$, then Theorem 8.1 in \cite{Kyprianou} tells us that
\beqq
\e_x\left[e^{-q\tau_a^+} {\mathbf 1}_{\{\tau_a^+<\tau_0^-\}}\right]=\frac{W^{(q)}(x)}{W^{(q)}(a)}, \;\;\; x\le a, \; q\ge 0. 
\eeqq
In fact, this identity justifies the name ``scale function'': we see that $W^{(q)}(x)$ plays an analogous role to scale function for diffusions. 

Our main goal in this section is to obtain an expression for the scale function $W^{(q)}(x)$ in terms of the Laplace exponent and the roots 
of the entire function $\psi_Y(z)-q$. We will consider spectrally negative L\'evy processes, whose L\'evy measure satisfies the following definition. 

\begin{definition}\label{def_regular_spec_negative}
 We say that the L\'evy measure of a spectrally negative process $Y$ is {\it regular} if 
there exists $n \in {\mathbb N}\cup\{0\}$ such that 
\begin{itemize}
 \item[(a)] for some constants $C, \alpha$ and $\{C_j,\alpha_j\}_{1\le j \le m}$  we have
\beqq
\pibar^-(x)-C|x|^{-\alpha}-\sum\limits_{j=1}^{m} C_j |x|^{-\alpha_j} \in {\mathcal {PC}}^{n+1}[-k,0], 
\eeqq
where  $\alpha, \alpha_j \in (-\infty,1) \cup (1,2)$ and $\alpha_j<\alpha$;
 \item[(b)] $\pibar^-{}^{(n)}(-k^+) \ne 0$;
 \item[(c)] $\pibar^-(x)\in {\mathcal C}^{n-1}(\r^-)$ (this condition is not needed for $n=0$).
\end{itemize}
\end{definition}

By considering the dual process $\hat Y=-Y$ and using Proposition \ref{prop_asymptotic_psi}  we see that if $Y$ has a regular L\'evy measure then 
its Laplace exponent satisfies 
\beq\label{psi_Y_asymptotics}
\psi_Y(-z)=Ae^{kz}z^{-a}+Bz^b+ o\left(e^{kz}z^{-a}\right) + o \left(z^b\right)
\eeq
as $z\to \infty$ in the domain ${\mathcal Q}_1$, where $a\ge 0$ and $b>0$.  
Moreover, the parameters in the asymptotic expression (\ref{psi_Y_asymptotics}) are given by
\beq\label{eqn_AaBb}
 \begin{cases}
& A=\bar \Pi^- {}^{(n)}(-k^+), \; \textnormal{ and } \; a=n,  \\
 &  B=\frac{\sigma^2}2, \; \textnormal{ and } \; b=2, \; \textnormal{ if } \sigma>0, \\
 &  B=-\mu, \; \textnormal{ and } \; b=1, \; \textnormal{ if } \sigma=0 \textnormal{ and } \alpha<1, \\
 &  B=-C e^{-\pi \i \alpha} \Gamma(1-\alpha), \; \textnormal{ and } \; b=\alpha, \; \textnormal{ if } \sigma=0 \textnormal{ and } \alpha>1.
\end{cases}
 \eeq

Theorems \ref{thm_main} and \ref{thm_asymptotics} tell us that equation $\psi_Y(-z)=q$ has infinitely many solutions $\{\zeta_n\}_{n\ge 0}$ in ${\mathcal Q}_1$,  moreover
the first solution $\zeta_0$ is real and positive, we have $\zeta_n \ge \zeta_0$ for $n\ge 1$ and the large solutions $\zeta_n$ 
satisfy asymptotic relation (\ref{zeta_asymptotics}) with constants $A$, $a$, $B$ and $b$ as given in (\ref{eqn_AaBb}). 

The next Theorem is our main result in this section. It provides a series representation for the scale function $W^{(q)}(x)$ in terms of the Laplace exponent $\psi_Y(z)$ and the numbers $\zeta_n$. Its proof can be found in Section \ref{sec_proofs}.

\begin{theorem}\label{thm_w^q}
Assume that $q>0$ or $q=0$ and $\Phi(q)>0$. If the L\'evy measure of $Y$ is regular and all solutions to $\psi_Y(z)=q$ are simple, then for $x>0$ we have
 \beq\label{eqn_W^q}
 W^{(q)}(x)=\frac{e^{\Phi(q)x}}{\psi_Y'(\Phi(q))}+\frac{e^{- \zeta_0 x}}{\psi_Y'(-\zeta_0)}+
 2\sum\limits_{n\ge 1} \re \left[ \frac{e^{- \zeta_n x}}{\psi_Y'(-\zeta_n)} \right],
 \eeq
where the series converges uniformly on $[\epsilon,\infty)$ for every $\epsilon>0$. 
\end{theorem}

Formula \eqref{eqn_W^q} is in fact very similar to the corresponding expression for the scale function for meromorphic processes, see 
\cite{KuMo} and \cite{KuKyRi}. 

In the very unlikely case that some of the solutions to $\psi_Y(z)=q$ are not simple (see the discussion on page \pageref{discussion_simple_zeros}) 
the expression in the right-hand side of (\ref{eqn_W^q}) would have to be modified. The coefficients $\exp(- \zeta_n x)/\psi_Y'(-\zeta_n)$ are the residues of $\exp(zx)/(\psi_Y(z)-q)$ at $z=-\zeta_n$ (see the proof of Theorem \ref{thm_w^q} in Section \ref{sec_proofs}), 
and these coefficients would have to be appropriately modified if $\zeta_n$ is a root of $\psi_Y(-z)-q$ of multiplicity greater than one.

\section{Numerical examples}\label{sec_numerics}

The main reason why we are interested in the analytical structure of the Wiener-Hopf factorization is that its understanding can lead to efficient numerical algorithms for computing such important objects as the distribution of supremum $S_{\ee(q)}$ or infimum $I_{\ee(q)}$, or the scale function $W^{(q)}(x)$. These are not easy problems for general L\'evy processes. For example, computing the distribution of $S_{\ee(q)}$ in general involves evaluating numerically two integral transforms: first one has to compute the positive Wiener-Hopf factor via the formula (see \cite{Mordecki})
\beqq
  \e\left[ e^{\i z S_{\ee(q)}} \right]=\exp\bigg[\frac{z}{2\pi {\rm i}} \int_{\r}  \ln\bigg(\frac{q}{q-\psi(\i u)}\bigg) \frac{{\rm d} u}{u(u-z)} \bigg], \qquad \im(z)>0,
 \eeqq
and then perform an inverse Fourier transform to recover the distribution of $S_{\ee(q)}$. Similarly, computing the scale function $W^{(q)}(x)$ in general is equivalent to inverting Laplace transform in \eqref{def_W^q} (see \cite{KuKyRi} for the detailed discussion and comparison of several numerical algorithms for computing the scale function).

The results presented in this paper lead to quite different approach for computing the distribution of $S_{\ee(q)}$ and the scale function $W^{(q)}(x)$. This approach does not rely on the numerical evaluation of multiple integral transforms; instead, the main ingredients are the solutions
to the equation $\psi(z)=q$ and infinite series representations \eqref{distribution_S_ee(q)} and \eqref{eqn_W^q}. Therefore, 
this approach is very close in spirit to the techniques that are used for
processes with positive jumps of rational transform \cite{Mordecki} or meromorphic \cite{KuzKypJC} processes.

Our main goal in this section is to give a brief  description of the numerical algorithms and techniques which 
are suitable for L\'evy processes with bounded positive jumps, in particular, we would like to show that series expansions
\eqref{distribution_S_ee(q)} and \eqref{eqn_W^q} may lead to efficient numerical computations. 
However, the detailed investigation of these numerical algorithms is beyond the scope of the current paper and 
we will leave to future work such important questions as the speed of convergence, rigorous error analysis, etc.

\subsection{Preliminaries}\label{subsec_preliminaries}

In order to implement the numerical algorithms based on formulas \eqref{distribution_S_ee(q)} and \eqref{eqn_W^q} we have to solve the 
following problem: how can we find the complex solutions of the equation $\psi(z)=q$? As we will see, 
it is not an easy problem, yet it can be solved rather efficiently provided that we use the right techniques.

The main problem in finding the zeros of $\psi(z)-q$ is that all of them (except for $\zeta_0$) are complex numbers. Note that the real zero $\zeta_0$ can be easily found by bisection method followed by Newton's method. 
The large zeros of $\psi(z)=q$ satisfy asymptotic relation (\ref{zeta_asymptotics}), thus they can be found by Newton's method which is started from the value in the right-hand side of (\ref{zeta_asymptotics}). The only problem that remains is how to compute the complex zeros of $\psi(z)-q$ which are not too large. 

The problem of computing the zeros of an analytic function inside a bounded domain has been investigated by many authors, see \cite{Dellnitz2002325,Yakoubsohn,Ying_Katz} and the references therein. 
We will follow the method presented in \cite{Dellnitz2002325}, which is based on Cauchy's Argument Principle. 
Let us recall this important result. We denote the change in the argument of an analytic function $f(z)$ over a piecewise smooth curve $C$ as 
\beq\label{def_arg}
\Delta \arg(f,C)=\im \left[ \int_C \frac{f'(z)}{f(z)}\d z \right],
\eeq
provided that $f(z)\ne 0$ for $z\in C$. 
Cauchy's Argument Principle states that if $C$ is a simple closed contour which is oriented counter-clockwise, 
and $f(z)$ is analytic and non-zero on $C$ and analytic inside $C$, then $\Delta \arg(f,C)=2 \pi N$, 
where $N$ is the number of zeros of $f(z)$ inside contour $C$.

This result leads to a practical iterative procedure to determine the complex zeros of an analytic function $f(z)$ inside a closed contour.  We start with an initial rectangle $R$ in the complex plane and proceed through the following sequence of steps: (i) compute $N=N(R)$ - the number of zeros
of $f(z)$ inside rectangle $R$,
(ii) if $N=0$, then we stop, (iii) if $N=1$ - we try to find the zero using Newton's method started from a point inside the rectangle, (iv) if
$N>1$ or if the Newton's method in step (iii) fails - then we subdivide the rectangle $R$ into a finite number of disjoint rectangles $R_j$. For each of the smaller rectangles $R_j$ we proceed through the same sequence of steps (i)$\to$(ii)$\to$(iii)$\to$(iv). At some point 
the rectangles which contain zeros of $f(z)$ become very small, therefore the starting point of the Newton's method is close to the target and 
the Newton's method converges to the zero of $f(z)$. This shows that in theory we should be able to recover all zeros of $f(z)$ inside $R$. 
There are some technical issues which arise when at some step of the algorithm we obtain a rectangle with one of the zeros being extremely close to the boundary of this rectangle, however we found that in our numerical experiments this issue was not a big problem. 
More details of this algorithm and some numerical examples can be found in \cite{Dellnitz2002325}.

Next, let us review some basic facts about the incomplete gamma function, which will be used extensively in this section. The incomplete gamma function is defined for $\re(s)>0$ and $z>0$ as
\beq\label{def_gamma}
\gamma(s,z)=\int\limits_0^z u^{s-1} e^{-u} \d u.
\eeq
It is known (see section 8.35 in \cite{Jeffrey2007}) that the function $z \mapsto z^{-s}\gamma(s,z)$ is an entire function and it can be represented by 
a Taylor series (which converges everywhere in the complex plane) as follows
\beq\label{eqn_inc_gamma}
z^{-s} \gamma(s,z)=s^{-1}  {}_1F_1(s,s+1;-z)=\sum\limits_{n\ge 0} \frac{(-1)^n}{n!} \frac{z^n}{s+n}=e^{-z} \sum\limits_{n\ge 0}  \frac{z^n}{(s)_{n+1}}.
\eeq
Here, as usual, $(a)_n=a(a+1)\dots(a+n-1)$ denotes the Pochhammer symbol and the function ${}_1F_1(a,b;z)$ which appears in the above equation is the confluent hypergeometric function 
\beq\label{def_1F1}
{}_1 F_1(a,b;z)=\sum\limits_{n\ge 0} \frac{(a)_n}{(b)_n} \frac{z^n}{n!},
\eeq
see chapter 6 in \cite{Erdelyi1955V3} for an extensive collection of results and formulas related to this function.

\begin{figure}
\centering
\subfloat[][]{\label{fig_proof_a}\includegraphics[height =5cm]{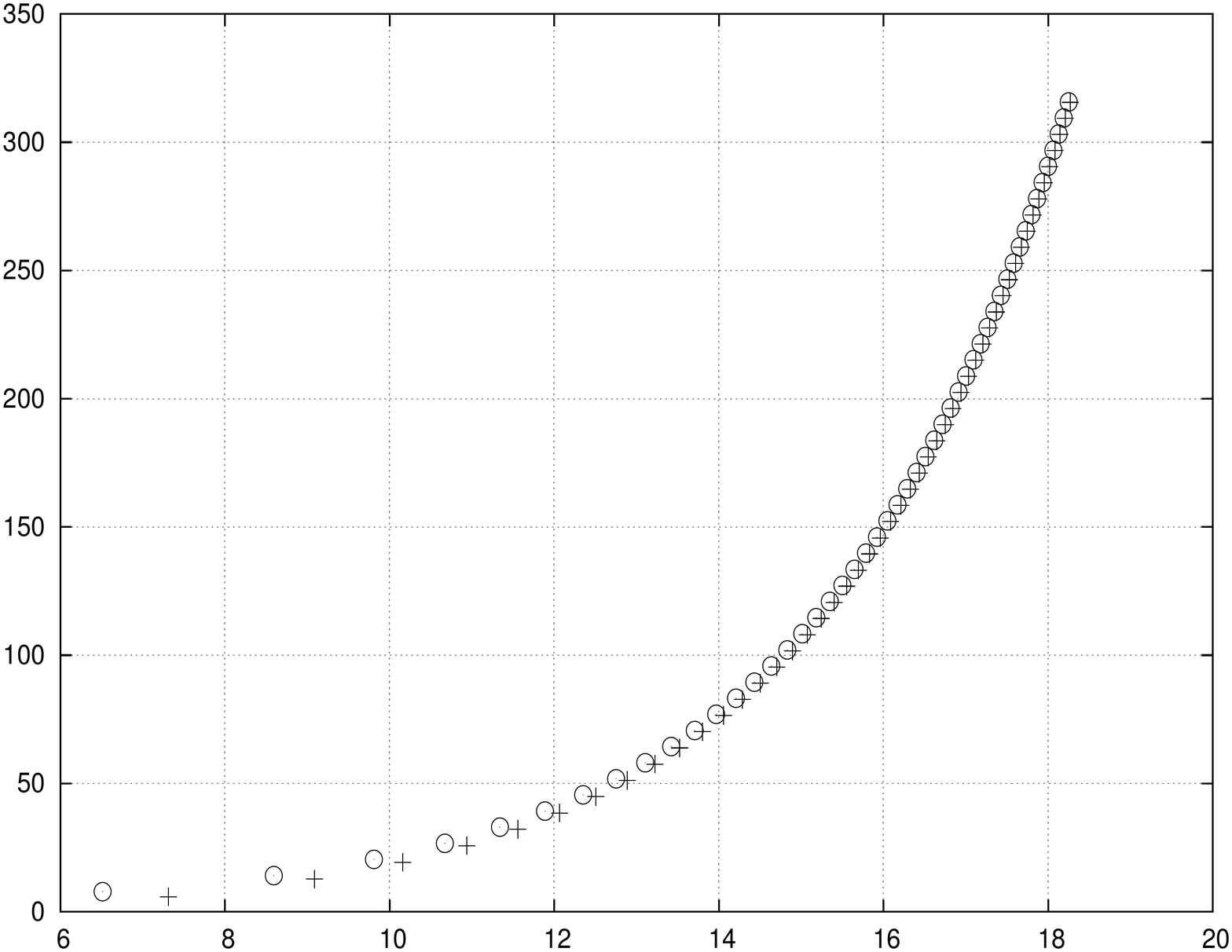}} 
\subfloat[][]{\label{fig_proof_b}\includegraphics[height =5cm]{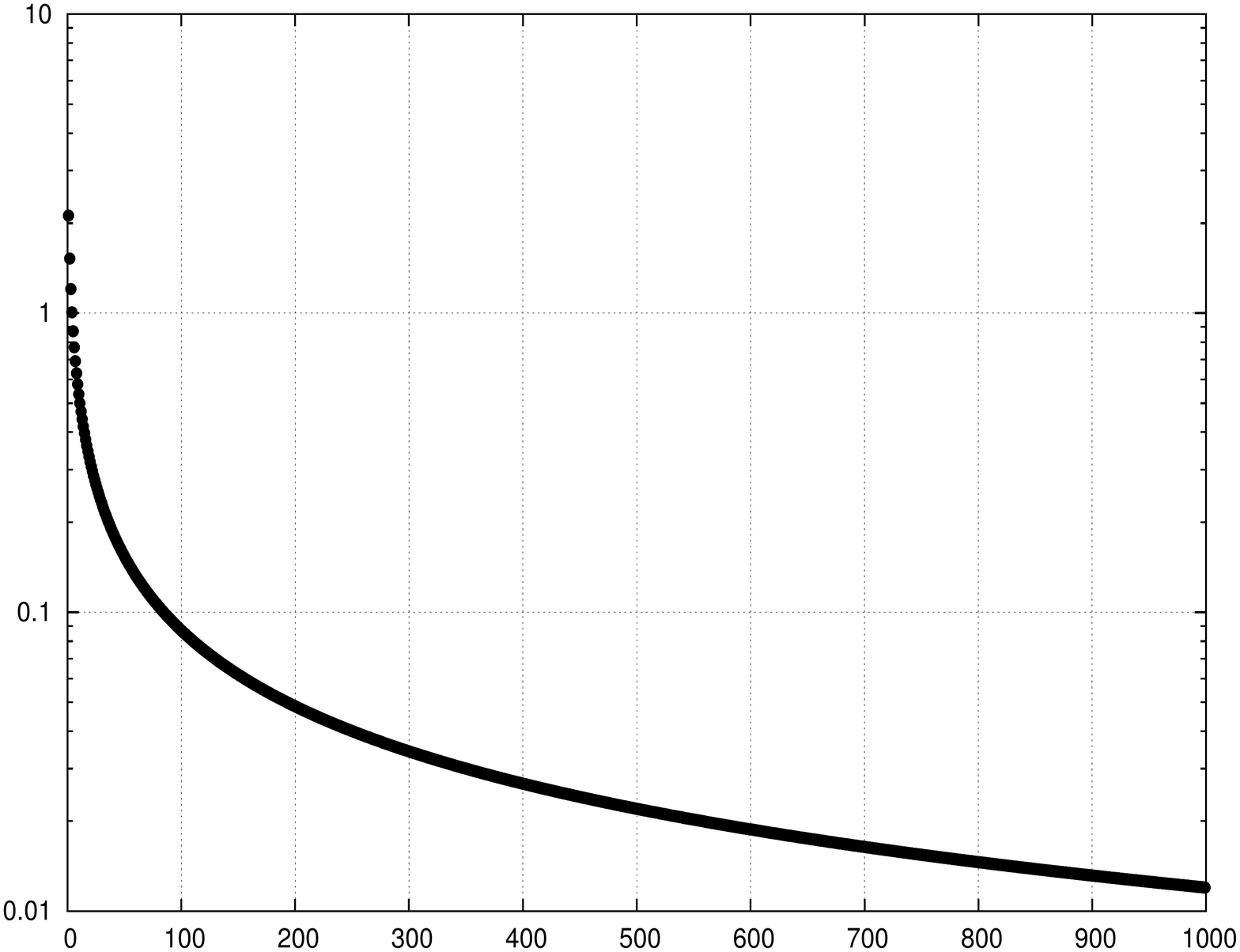}} 
\caption{On graph (a) we present the first fifty roots $\zeta_n$ (crosses) and their approximations (circles) given by
Theorem \ref{thm_asymptotics}. On graph (b) we present the distance from $\zeta_n$ to its approximation (note that we use logarithmic scale for the y-axis).} 
\label{fig_roots}
\end{figure}

While the incomplete gamma function is not one of the elementary functions, it can still 
be easily evaluated everywhere in the complex plane. In fact, numerical routines for evaluating this function are 
provided in such computational software programs 
as Maple and Mathematica. 
One should use different strategies for computing $\gamma(s,z)$ depending on whether $|s|$ and/or $|z|$ is large. However, we will only need to compute $\gamma(s,z)$ for a fixed $s$ (which is real and not large) and for various values of $z$. In this case the numerical 
algorithms are based on one of the infinite series expansions presented (\ref{eqn_inc_gamma}) when $|z|$ is not large,
 or on the various asymptotic approximations (and expansions in continued fractions, such as formula 
 8.358 in \cite{Jeffrey2007}) when $|z|$ is large. 
We refer to \cite{Jones1985401} or  \cite{Winitzki03} for all the details. 

Finally, we would like to mention that the code for all numerical experiments was written in C++ and the computations were performed on a standard laptop (Intel Core i5 2.6 GHz processor and 4 GB of RAM).

\subsection{Numerical example 1: processes with double-sided jumps}\label{subsec_numerics_example1}

For our first numerical experiment, we consider a generalized tempered stable process (see \cite{Cont}), also known as KoBoL process, with the L\'evy measure truncated at a positive number $k$
\beq\label{def_pi_temp_stable}
\Pi(\d x)={\mathbf 1}_{\{x<0\}} \hat C \hat \alpha e^{\hat \beta x} |x|^{-1-\hat \alpha} \d x+
 {\mathbf 1}_{\{0<x<k\}}  C  \alpha e^{- \beta x} x^{-1- \alpha} \d x.
\eeq

\begin{proposition}\label{Laplace_exponent_X}
Let $X$ be a L\'evy process, with the L\'evy measure $\Pi(\d x)$ given by (\ref{def_pi_temp_stable}). Then the Laplace exponent of $X$
is given by
\beq\label{def_psi_X}
\psi(z)=\frac12 \sigma^2 z^2 +\mu z - \hat C \Gamma(1-\hat\alpha) (\hat \beta+z)^{\hat \alpha}+ C \alpha (\beta-z)^{\alpha} \gamma(-\alpha,k(\beta-z))+\eta,
\eeq
where $\eta$ is chosen so that $\psi(0)=0$. 
\end{proposition}

The proof of Proposition \ref{Laplace_exponent_X} can be found in Section \ref{sec_proofs}. Note that when $z< \beta$
and $k\to \infty$, 
then $\gamma(-\alpha,k(\beta-z)) \to \Gamma(-\alpha)$ (this follows from (\ref{def_gamma}) when $\alpha<0$, see also formulas (8.356.3) and 
(8.357.1) in \cite{Jeffrey2007}). This confirms the intuitively obvious result that as the cutoff  $k$ becomes very large, the Laplace exponent of the truncated process converges to the Laplace exponent of the generalized tempered stable process, see Proposition 4.2 in \cite{Cont}. 
Note also that the function
$ (\beta-z)^{\alpha} \gamma(-\alpha,k(\beta-z))$ which appears in (\ref{def_psi_X}) is an entire function of $z$, 
which confirms the fact that $\psi(z)$ is analytic in the half-plane $\re(z)>0$. 
As in the case of generalized tempered stable processes, when $\sigma>0$ or $\alpha>1$ or $\hat \alpha>1$ we have a process of infinite variation. Finally, in the case when $\sigma=0$ and $\alpha<1$ and $\hat \alpha<1$  we have a process of finite variation, and in this case $\mu$ corresponds to the linear drift of this process.

\begin{figure}
\centering
\includegraphics[height =6cm]{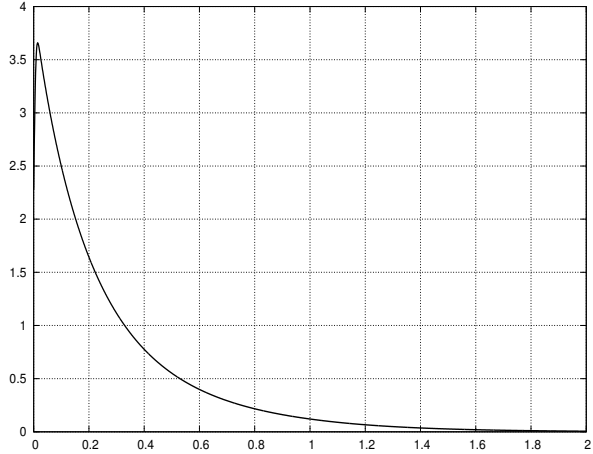} 
\caption{The density of $S_{\ee(q)}$ for $q=1$.} 
\label{fig_density_Sq}
\end{figure}

We consider the following parameters:
\beqq
\sigma=1, \;\; \mu=-2,  \; \; C=\hat C=1, \;\; \alpha=\hat \alpha=0.5, \;\; \beta=1, \; \; \hat \beta=2, \;\; k=1.
\eeqq
These parameters give us a process with negative linear drift and finite variation infinite activity jumps. 
Note that since $\sigma=1$ we have a process with paths of unbounded variation. We also 
set $q=1$. 

First we compute 1000 roots $\zeta_n$ using the method discussed in Section \ref{subsec_preliminaries}. 
Overall, it takes just 0.15 seconds to compute 1000 roots. The results are presented on Figure \ref{fig_roots}. Note that, as expected, the approximation to $\zeta_n$ provided by \eqref{zeta_asymptotics} becomes better and better as $n$ increases,
but is not so good for small values of $n$.

Next, we compute the density $p(x)$ of the supremum $S_{\ee(q)}$ using the series representation \eqref{distribution_S_ee(q)}. The results 
are presented on Figure \ref{fig_density_Sq}. After we have pre-computed and stored the roots $\zeta_n$, 
computing 2000 values of $p(x)$ takes just 0.26 seconds. In order to test the accuracy we have numerically computed the integral of $p(x)$ on the interval $[0,10]$ (which should be close to $\p(S_{\ee(q)}>0)=1$). The result is equal 0.985 if we use 1000 roots and 0.995 if we use 5000 roots.

\subsection{Numerical example 2: a family of spectrally negative processes}\label{subsec_numerics_example2}

For our second numerical experiment we consider a spectrally negative L\'evy process $Y$, with the L\'evy measure defined as follows
\beqq
\Pi_Y(\d x)={\mathbf 1}_{\{-k<x<0\}} C  \alpha e^{ \beta x} |x|^{-1- \alpha} \d x.
\eeqq
Note that the dual process $\hat Y=-Y$ belongs to the class of processes described in Proposition
\ref{Laplace_exponent_X}, therefore the Laplace exponent of $Y$ is given by  
\beq
\psi_Y(z)=\frac12 \sigma^2 z^2 +\mu z + C \alpha (\beta+z)^{\alpha} \gamma(-\alpha,k(\beta+z))+\eta,
\eeq
where again $\eta$ is chosen so that $\psi_Y(0)=0$.

\begin{figure}
\centering
\subfloat[][$\sigma=0$]{\label{fig_proof_a}\includegraphics[height =5cm]{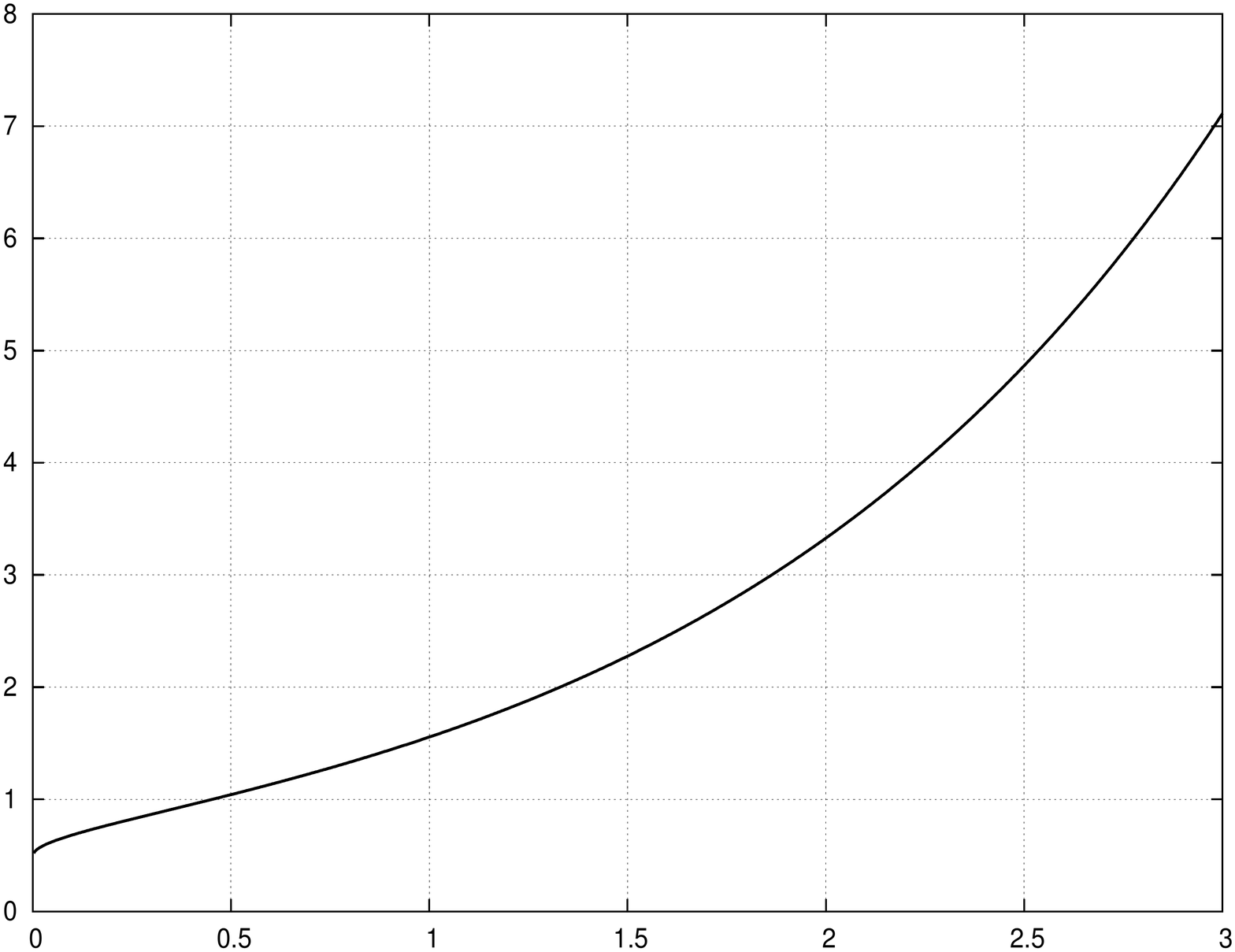}} 
\subfloat[][$\sigma=1$]{\label{fig_proof_b}\includegraphics[height =5cm]{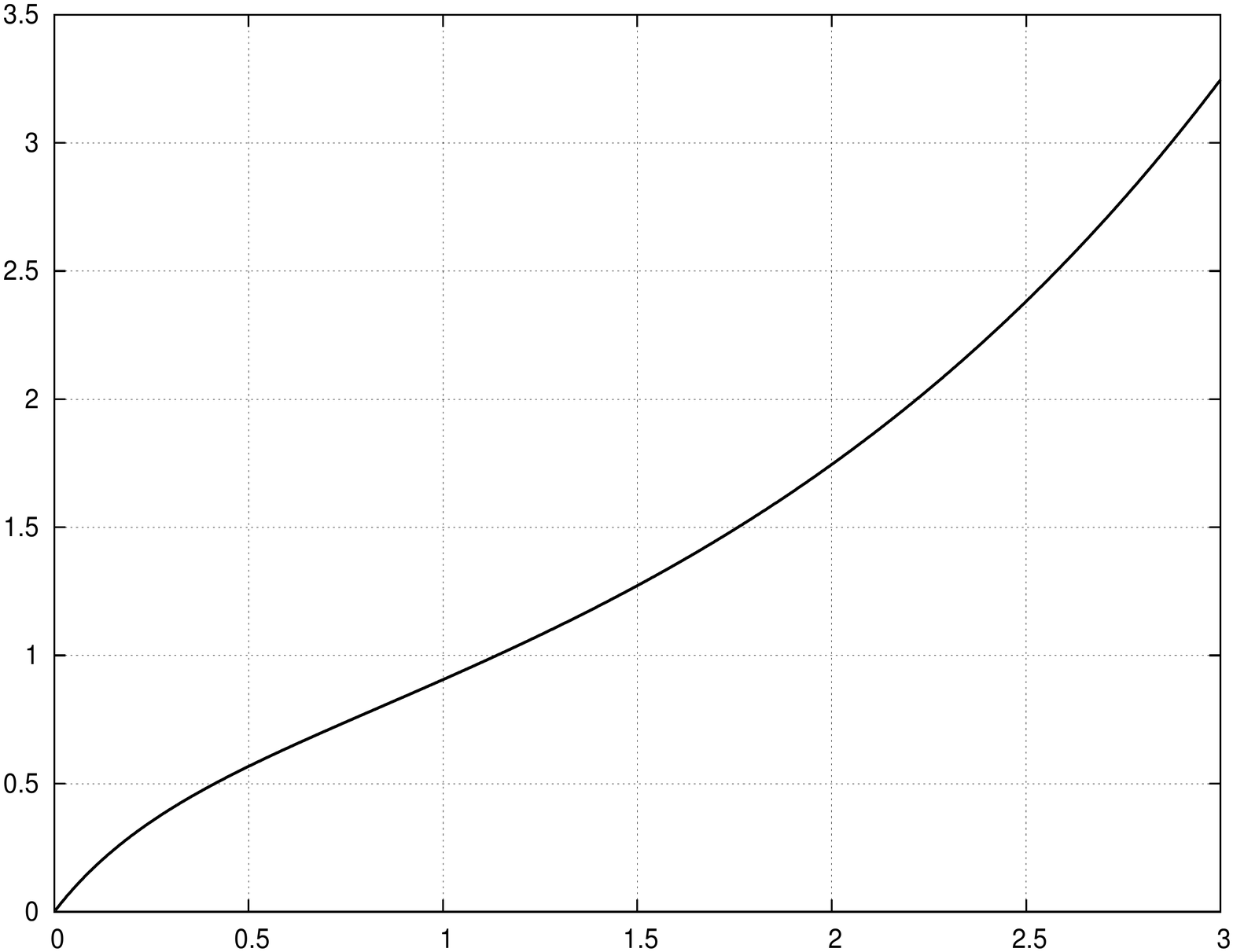}} 
\caption{The scale function $W^{(q)}(x)$ for the parameter set \eqref{par_set2}} 
\label{fig_Wq}
\end{figure}

We fix the following values of parameters
\beq\label{par_set2}
\sigma\in\{0,1\}, \;\; \mu=2,  \; \; C=1, \;\; \alpha=0.5, \;\; \beta=1, \;\; k=1,
\eeq
which define a spectrally-negative process with infinite activity/finite variation jumps and paths of 
finite/infinite variation depending on whether $\sigma=0$ or $\sigma=1$. We compute 1000 numbers $\{\zeta_n\}$ using the algorithm presented in Section \ref{subsec_preliminaries}. This computation takes 0.06 seconds. The qualitative behavior of the roots $\zeta_n$ is very similar to the one presented on Figure \ref{fig_roots}.  Computing the scale function via series representation \eqref{eqn_W^q} is also very fast: it takes just 0.07 seconds to compute 1000 values of $W^{(q)}(x)$ for equally spaced points $x\in [0,3]$. 

The results of computations are presented on Figure \ref{fig_Wq}. From Lemma 8.6 in \cite{Kyprianou} we know that $W^{(q)}(0)=0$ if the process $Y$ has infinite variation, and $W^{(q)}=1/\mu$ if the process has bounded variation (where $\mu$ is the linear drift). One can see from Figure \ref{fig_Wq} that our numerical results are in perfect agreement with the theoretical prediction. It would also be very interesting 
to compare the accuracy and performance of this algorithm for evaluating $W^{(q)}(x)$ with the methods used in \cite{KuKyRi}, 
however we have decided to leave this for future work.

\section{Proofs}\label{sec_proofs}

\begin{lemma}\label{lemma_support}
 Let $\nu(\d x)$ be a finite positive measure such that $\nu([a,+\infty))=0$ for some $a\in \r$. Define $k=\inf\{ a \in \r: \nu([a,+\infty))=0\}$. 
Then for every $\epsilon>0$ there exists $\xi=\xi(\epsilon)>0$ such that for all $z>\xi$ we have
\beq\label{eqn_lemma_support}
e^{(k-\epsilon)z}<\int\limits_{\r} e^{zx} \nu(\d x) < e^{(k+\epsilon)z}.
\eeq
\end{lemma}
\begin{proof}
 Assume that $\epsilon>0$. Since $\nu((k,\infty))=0$ we have for all $z>0$
 \beqq
 e^{-(k+\epsilon)z}\int\limits_{\r} e^{zx} \nu(\d x)= \int\limits_{(-\infty,k]} e^{z (x-k-\epsilon)} \nu(\d x) < e^{-\epsilon z} \nu (\r),
 \eeqq 
and the right-hand side in the above inequality goes to zero as $z \to +\infty$. This proves the upper bound in (\ref{eqn_lemma_support}). Similarly, 
 \beqq
 e^{-(k-\epsilon)z}\int\limits_{\r} e^{zx} \nu(\d x)>  e^{-(k-\epsilon)z} \int\limits_{[k-\epsilon/2,k]} e^{z x} \nu(\d x) > e^{ \frac{\epsilon}2 z} \nu ([k-\epsilon/2,k]).
 \eeqq 
 According to our definition of $k$, the quantity $\nu ([k-\epsilon/2,k])$ is strictly positive, thus the right-hand side in the above inequality 
goes to $+\infty$ as $z\to +\infty$, which proves the lower bound in (\ref{eqn_lemma_support}). 
\end{proof}


{\it Proof of Theorem \ref{thm_main}:} Let us prove (i). 
L\'evy-Khintchine formula (\ref{Levy_Khinthine2}) implies that the function $\psi(z)$ is analytic in the half-plane $\re(z)>0$ and convex for $z>0$. 
Since $\psi(0)=0$ and $\psi(z)$ increases exponentially as $z\to +\infty$ we conclude that there exists a unique simple real solution to $\psi(z)=q$, 
which we will denote $\zeta_0$. 

Let us prove that there are no other solutions in the vertical strip $0 \le \re(z)<\zeta_0$. From the definition of the Laplace exponent we find
\beqq
e^{t \re(\psi(z))}=\big | \e\left[ e^{z X_t} \right] \big| \le \e\left[ e^{\re(z) X_t} \right]=e^{t \psi(\re(z))}.
\eeqq
This shows that for $0\le \re(z) < \zeta_0$ we have $\re(\psi(z)-q)\le \psi(\re(z))-q<0$, therefore $\psi(z)\ne q$ in this vertical strip and
 we have proved the first part of Theorem \ref{thm_main}. 

Next, let us prove (ii), (iii) and (iv). Let us consider the ascending ladder process $(L^{-1}, H)$ and its Laplace exponent $\kappa(q,z)$. 
Similarly, let $\hat \kappa(q,z)$ denote the Laplace exponent of the descending ladder process $(\hat L^{-1}, \hat H)$, 
see section 6.2 in \cite{Kyprianou} for the definition and properties of these objects.  
Let $\Lambda(\d t, \d x)$ denote the L\'evy measure of the bivariate subordinator $(L^{-1}, H)$ (see section 6.3 in \cite{Kyprianou}). 
One can check that $\kappa(q,z)$ can be expressed in the following form  
\beq\label{thm_main_proof1}
\kappa(q,z)=\kappa(q,0)+a z-\int\limits_0^{\infty} \left( e^{-zx}-1\right) \Lambda^{(q)}(\d x)
\eeq 
where $a\ge 0$ and  
\beqq
\Lambda^{(q)}(\d x)=\int_0^{\infty} e^{-q t} \Lambda(\d t, \d x).
\eeqq
Formula (\ref{thm_main_proof1}) implies that the function $z\mapsto \kappa(q,z)$ is the Laplace exponent of a subordinator with 
the L\'evy measure $\Lambda^{(q)}(\d x)$ and drift $a$, which is killed at rate $\kappa(q,0)$. 
Note that $\Lambda^{(0)}(\d x)$ is the L\'evy measure of the ascending ladder height process $H$. The jumps in the process $H$ happen when the process $X$ 
jumps over the past supremum, thus it is clear that if the jumps of $X$ are bounded from above by $k$, then the same is true for the process $H$. 
Therefore $\Lambda^{(0)}((k,\infty))=0$, and since for each Borel set $B$ the quantity $\Lambda^{(q)}(B)$ is decreasing in $q$ 
we conclude that $\Lambda^{(q)}((k,\infty))=0$ for all $q\ge 0$. Therefore we have proved that the support of the measure $\Lambda^{(q)}(\d x)$ lies inside the interval $(0,k]$. 

Let us define $\tilde k=\inf\{ x>0: \Lambda^{(q)}((x,\infty))=0\}$. Note that $\tilde k \le k$, since we have established already that  
the support of the measure $\Lambda^{(q)}(\d x)$ lies inside the interval $(0,k]$. Using formula (\ref{thm_main_proof1}) and the fact that $\Lambda^{(q)}(\d x)$ has finite support we conclude that $\kappa(q,z)$ is an entire function of $z$. 


Using the Wiener-Hopf factorization (see Theorem 6.16 in \cite{Kyprianou}) we find that for $\re(z)\le 0$
\beq\label{WH_fact_1}
\bigg| \frac{\kappa(q,0)}{\kappa(q,-z)} \bigg|=\big | \e\left[ e^{z S_{\ee(q)}} \right] \big| <  \e\left[ e^{\re(z) S_{\ee(q)}} \right]\le 1,
\eeq
this shows that the function $\kappa(q,-z)$ has no zeros in the half-plane $\re(z)\le 0$. By the same argument we conclude that $\hat \kappa(q,z)$ has 
no zeros in the half-plane $\re(z) \ge 0$. Therefore, using 
the Wiener-Hopf factorization $q-\psi(z)=\kappa(q,-z) \hat \kappa(q,z)$ we find that  all zeros of $\kappa(q,-z)$ in 
the half-plane $\re(z)>0$ coincide with the zeros of $q-\psi(z)$. Recall that we have labeled these zeros as $\{\zeta_0,\zeta_n,\bar \zeta_n\}_{n\ge 1}$, where
$\zeta_n \in {\mathcal Q}_1$ are arranged in the order of increase of absolute value.

Next, we use Proposition 2 in \cite{Bertoin} and the fact that $\kappa(q,z)$ is the Laplace exponent of a subordinator to conclude that 
$\kappa(q,\i z) = O(z)$ as $z\to \infty$, $z\in \r$. Therefore the following integral converges
(which is equivalent to saying that $\kappa(q,\i z)$ belongs to Cartwright class of entire functions)
\beqq
\int\limits_{\r} \frac{\max(\ln(|\kappa(q,\i z)|),0)}{1+z^2} \d z,
\eeqq 
and we can apply Theorem 11, page 251 in \cite{Levin1980} (see also remark 2, page 130 in \cite{Levin1996}) to conclude 
that $\kappa(q,-z)$ can be factorized as 
\beq\label{thm_main_proof2}
\kappa(q,-z)=\kappa(q,0) e^{\frac{\tilde k}2 z}  \left(1-\frac{z}{\zeta_0} \right) \prod\limits_{n\ge 1} \left(1-\frac{z}{\zeta_n} \right)\left(1-\frac{z}{\bar\zeta_n} \right).
\eeq
Moreover, all of the roots $\zeta_n$ (except for a set of zero density) lie inside arbitrarily small angle 
$\pi/2-\epsilon<\arg(z)<\pi/2$, the density of the roots in this angle exists and is equal to $\tilde k/(2\pi)$ and the series
$\sum\re\left(\zeta_n^{-1} \right)$ converges. 

Using (\ref{thm_main_proof1}) and the following result
\beqq
\e\left[ e^{-zS_{\ee(q)}} \right]=\phiqp(\i z)=\frac{\kappa(q,0)}{\kappa(q,z)},
\eeqq
(see Theorem 6.16 in \cite{Kyprianou}) we obtain formula  (\ref{wh_factor}) for the Wiener-Hopf factors, and in order to finish the proof 
we only have to show that $\tilde k= k$. This fact seems to be intuitively clear, as it means that the upper boundary of the support of the L\'evy measure of the 
process $X$ is exactly equal to the upper boundary of the support of the L\'evy measure of the ascending ladder height process $H$. However, we were not able 
to find a simple probabilistic argument to prove this statement, and we will use an analytic approach instead. 
Assume that $\tilde k<k$ and define $\epsilon=(k-\tilde k)/3$. 
Using the L\'evy-Khintchine formula (\ref{Levy_Khinthine2}) and Lemma \ref{lemma_support} one can check that  $|\psi(z)|>\exp((k-\epsilon)z)$ for all $z>0$ large enough. 
Similarly, using (\ref{thm_main_proof1}), Lemma \ref{lemma_support} and the fact that $\Lambda^{(q)}(\d x)$ has support on $(0,\tilde k]$ we can check that 
$|\kappa(q,-z)|<\exp((\tilde k+\epsilon)z)$ for all $z>0$ large enough. 
From the Wiener-Hopf factorization $\hat \kappa(q,z)=(q-\psi(z))/\kappa(q,-z)$ we conclude that
\beqq
|\hat \kappa(q,z)|> \frac{e^{(k-\epsilon)z}}{e^{(\tilde k+\epsilon)z}}=e^{\epsilon z}
\eeqq
for all $z>0$ large enough. But this is not possible, as we know that $\hat \kappa(q,z)$ is the Laplace exponent of a subordinator, thus 
$\hat \kappa(q,z)=O(z)$ as $z\to \infty$, $\re(z)\ge 0$ (see Proposition 2 in \cite{Bertoin}). Therefore the inequality $\tilde k<k$ is not true. At the same time
we must have  $\tilde k \le k$, since the support of the measure $\Lambda^{(q)}(\d x)$ lies inside the interval $(0,k]$. 
This implies that $\tilde k=k$ and ends the proof of parts (ii), (iii) and (iv) of Theorem \ref{thm_main}. 
\qed
\\


 Recall that $\Delta \arg(f,C)$, which was defined in  (\ref{def_arg}), denotes the change in the argument of $f(z)$ over a curve $C$. 
The following result will be used in the proof of Theorem \ref{thm_asymptotics}. 
\begin{proposition}\label{prop_Delta_arg_property}
 Assume that $f$ and $g$ are analytic on a piecewise curve $C$ and $f(z)\ne 0$ for $z\in C$. If for some $\epsilon \in (0,1)$ we have $|g(z)|<\epsilon |f(z)|$ for all $z\in C$, then
\beq\label{Delta_arg_property}
\big | \Delta \arg(f+g,C) - \Delta \arg(f,C) \big |< 4 \epsilon. 
\eeq
\end{proposition}
\begin{proof}
From the definition of $\Delta \arg(f,C)$  (\ref{def_arg}) it follows that
\beqq
\Delta \arg(f+g,C) = \Delta \arg(f,C)+\Delta \arg(1+g/f,C).   
\eeqq
Due to the condition $|g(z)/f(z)|<\epsilon$ for $z\in C$ we know that the set $\{w=1+g(z)/f(z): \; z \in C\}$ lies inside the circle of radius $\epsilon$ with center at one. Using elementary geometric considerations we check that  
\beqq
\max\{|\arg(w)|: \; |w-1|<\epsilon\}=\arcsin(\epsilon)<2\epsilon,
\eeqq 
thus the change of the argument of any curve lying inside the circle $|w-1|<\epsilon$ cannot be greater 
than $4\epsilon$, which proves (\ref{Delta_arg_property}).
\end{proof}


We will also need the following two facts:
\begin{itemize}
 \item[(i)] For any real $u>0$ and any complex number $v$ for which $|\arg(v)|<\pi/2$ we have
\beq\label{fact1} 
  |\arg(u+v)| \le |\arg(v)|.
\eeq
 \item[(ii)]  For any two complex numbers $u$, $v$ we have 
\beq\label{fact2}
|u+v|\ge  \cos(\arg(u)-\arg(v)) (|u|+|v|).
\eeq
\end{itemize}
 Both of the above facts can be easily verified by elementary geometric considerations


\vspace{0.5cm}
{\it Proof of Theorem \ref{thm_asymptotics}:}
Let us denote the set of solutions to $\psi(z)=q$ in ${\mathcal Q}_1$ as ${\mathcal Z}=\{\zeta_n\}_{n\ge 1}$ and introduce numbers
\beqq
z_n=\frac{1}{k}  \left[\ln \left( \bigg | \frac{B}{A} \bigg| \right)+(a+b) \ln \left(\frac{2n\pi }{k}\right) \right] +
\frac{\i }{k} \left[ \arg\left( \frac{B}{A} \right)+\left( \frac12 (a+b) + 2n+1\right) \pi \right].
\eeqq
First let us prove that every solution to the equation $\psi(z)=q$ which has sufficiently large absolute value must be close to one of $z_n$. 
From the asymptotic expansion
(\ref{psi_asymptotics}) we find that when $z\to \infty$, $z\in {\mathcal Z}$ we have 
\beq\label{psi_asymptotics_2}
 e^{kz}=-\frac{B}{A} z^{a+b} (1+o(1)) .
\eeq
Considering the absolute value of both sides  of (\ref{psi_asymptotics_2}) we conclude
\beq\label{re_z_asymptotics}
\re(z)=\frac{1}{k} \ln\bigg |\frac{B}{A} \bigg| + \frac{a+b}{k} \ln|z|+o(1), \;\;\; z\to \infty, \; z\in {\mathcal Z}.  
\eeq
From the above asymptotic expression it follows that $\re(z)=O(\ln|z|)$, which in turn implies
\beq\label{arg_z_asymptotics}
\arg(z)=\frac{\pi}2+o(1), \;\;\; z\to \infty, \; z\in {\mathcal Z}.  
\eeq

Next, considering the argument of both sides of (\ref{psi_asymptotics_2}), we find that 
\beqq
\arg\left( e^{k z} \right)=\arg\left( -\frac{B}{A} \right)+ \arg\left( z^{a+b}\right)+\arg(1+o(1)) \;\;\; (\mod \; 2 \pi),
\eeqq
and using \eqref{arg_z_asymptotics} we conclude that there exists an integer number $n$ such that 
\beq\label{im_z_asymptotics}
k\im(z)=\arg\left(\frac{B}{A}\right)+(2n+1)\pi + \frac12 (a+b) \pi + o(1).  
\eeq 

Finally, asymptotic expressions (\ref{re_z_asymptotics}) and (\ref{im_z_asymptotics}) imply that 
\beqq
\re(z)=\frac{1}{k}  \left[\ln \left( \bigg | \frac{B}{A} \bigg| \right)+(a+b) \ln \left(\frac{2n\pi }{k}\right) \right]+o(1),
\eeqq
which together with (\ref{im_z_asymptotics}) shows that every sufficiently large solution to $\psi(z)=q$ must be close to one
of the numbers $z_n$.

\begin{figure}
\centering
\subfloat[][]{\label{fig_proof_a}\includegraphics[height =7cm]{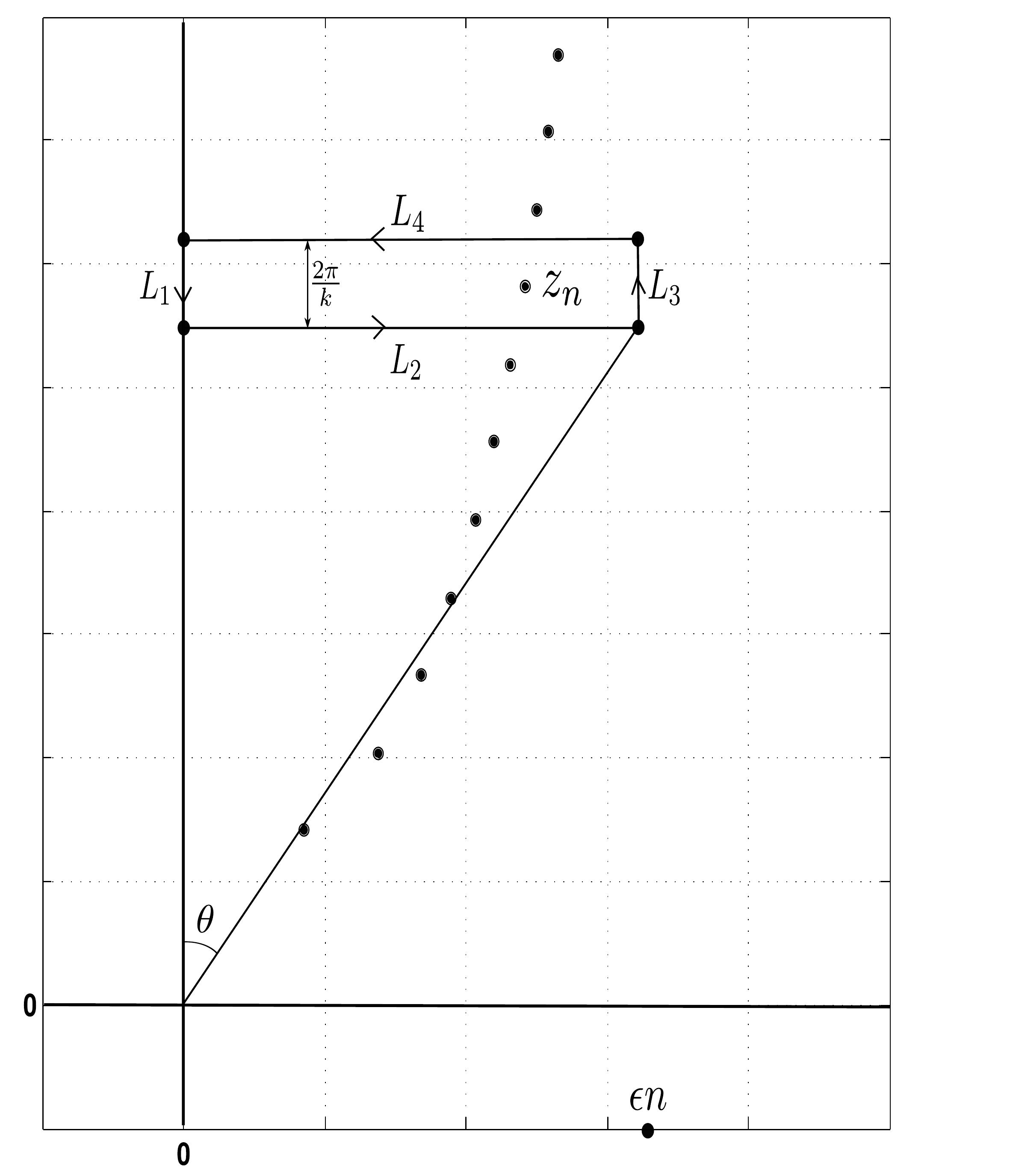}} 
\subfloat[][]{\label{fig_proof_b}\includegraphics[height =7cm]{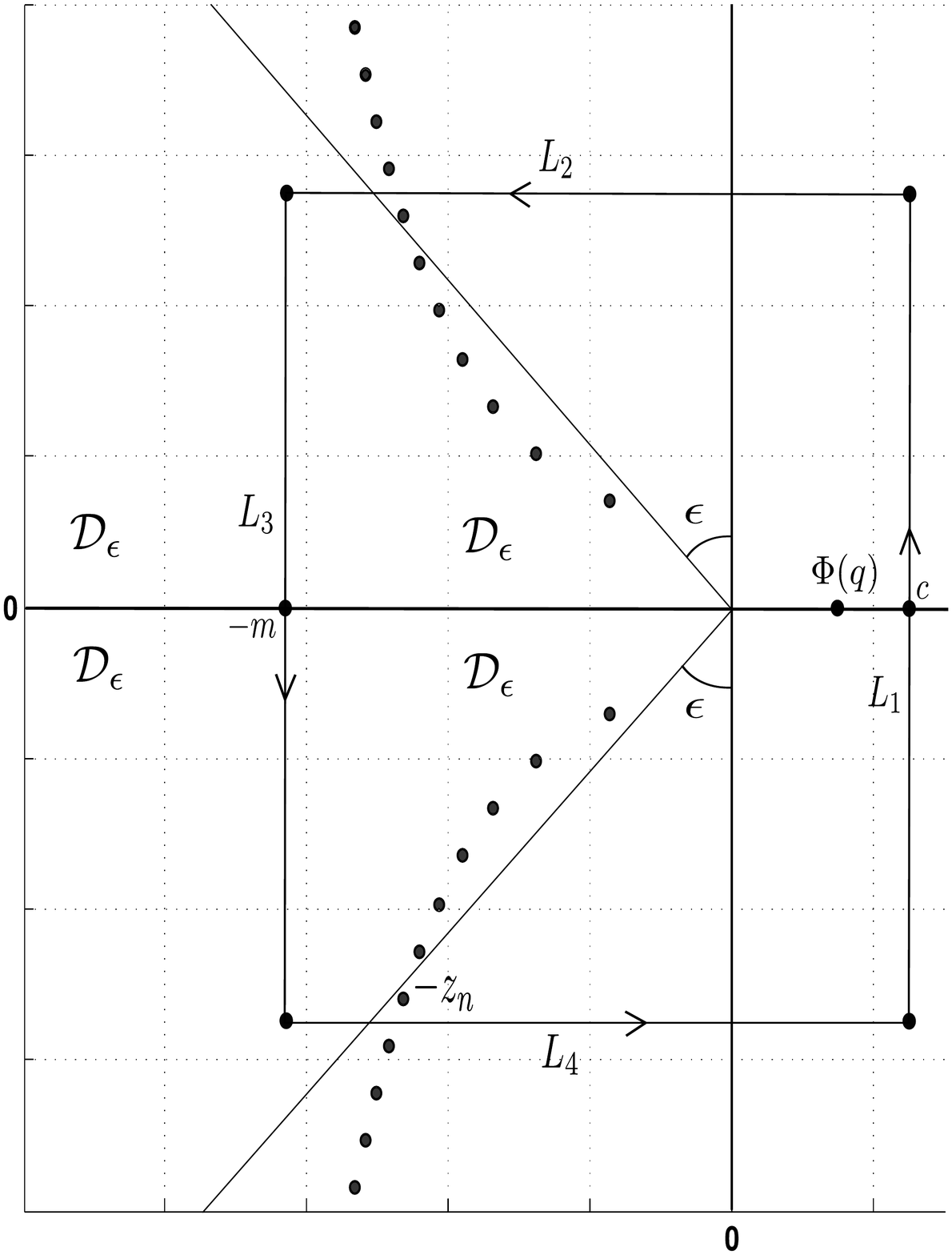}} 
\caption{Illustrations to the proofs of Theorem \ref{thm_asymptotics} and Theorem \ref{thm_w^q}.} 
\label{fig_Vdxds}
\end{figure}

Now let us prove the converse statement: for all $n$ large enough there is always a solution of the equation $\psi(z)=q$ near a point $z_n$. 
We set $\epsilon=1/(16k(a+b))$, assume that $n$ is a large positive integer and consider the following contour $L=L(n)=L_1 \cup L_2 \cup L_3 \cup L_4$, defined as
\beqq
 L_1&=&L_1(n)=\{ z\in \c: \; \re(z)=0, \; |\im(z)-\im(z_n)|\le \frac{\pi}{k}\}, \\
 L_2&=&L_2(n)=\{ z\in \c: \; \im(z)=\im(z_n)-\frac{\pi}{k}, \; 0\le \re(z) \le \epsilon n\}, \\
 L_3&=&L_3(n)=\{ z\in \c: \; \re(z)=\epsilon n, \; \; |\im(z)-\im(z_n)|\le \frac{\pi}{k}\}, \\
 L_4&=&L_4(n)=\{ z\in \c: \; \im(z)=\im(z_n)+\frac{\pi}{k}, \; 0\le \re(z) \le \epsilon n\}.
\eeqq
As we see on Figure \ref{fig_proof_a}, $L$ is a rectangle of dimensions $2\pi/k$ and $\epsilon n$, which contains exactly one point $z_n$ for $n$ large enough.
We assume that this contour is oriented counter-clockwise. Our goal is to prove that $\Delta \arg(\psi(z)-q,L(n))=2\pi$ for all $n$ large enough, and our strategy is 
to show that the change in the argument over $L_1$, $L_2$ and $L_4$ is small, while the change in the argument over $L_3$ is close to $2\pi$. 

First of all, it is clear that the number of zeros of $\psi(z)-q$ inside the contour $L$ is the same as the number of zeros of $F(z)=z^a (\psi(z)-q)$ inside the same contour. Asymptotic expression (\ref{psi_asymptotics}) tells us that
\beq\label{asymptotics_F(z)}
F(z)=A e^{kz}+ B z^{a+b} + o\left(e^{kz}\right)+o\left(z^{a+b}\right), \;\;\; z\to \infty, \; z \in {\mathcal Q}_1. 
\eeq

Let us first consider the interval $L_1$. Since $\arg(z)=\pi/2$ on this interval, we have $\Delta \arg(z^{a+b},L_1)=0$. 
Equation (\ref{asymptotics_F(z)}) implies that $F(z)=B z^{a+b} + o\left( z^{a+b} \right)$ when $z\in L_1$ and $n\to +\infty$, 
thus we use Proposition \ref{prop_Delta_arg_property} and conclude that
for all $n$ large enough we have $|\Delta \arg(F,L_1(n))|<1/4$.  

Let us consider the contour $L_2$. From the definition of this contour it follows that for all $z\in L_2$
\beq\label{contour_L2_1}
\arg\left(\frac{A}{B}\exp\left(kz-\frac{\pi \i}2(a+b)\right)\right)=0. 
\eeq
Also, looking at Figure \ref{fig_proof_a} one can check that 
$\pi/2-\theta \le \arg(z) \le \pi/2$ for all $z \in L_2$, where  we have defined
\beqq
\theta=\theta(\epsilon,n)=\arctan\left( \frac{\epsilon n}{\im(z_n)-\pi/k} \right).
\eeqq
Note that as $n\to +\infty$ we have $\theta(\epsilon,n)\to \arctan(\epsilon k /(2\pi))$, and the latter quantity is smaller than $k \epsilon $.  This implies that for all $n$ large enough we have $\theta(\epsilon,n)< k \epsilon$. Thus we have proved that for all $n$ large enough we have  $\pi/2-k \epsilon \le \arg(z) \le \pi/2$ when $z \in L_2$, which is equivalent to
\beq\label{contour_L2_2}
\bigg |\arg\left(z^{a+b}\exp\left(-\frac{\pi \i}2(a+b)\right)\right)\bigg|< k (a+b) \epsilon=\frac{1}{16}, \;\;\; z\in L_2. 
\eeq
 From (\ref{contour_L2_1}),  (\ref{contour_L2_2}) and property \eqref{fact1}  it follows that for all $z\in L_2$ the number
\beqq
w=\frac{A}{B}\exp\left(kz-\frac{\pi \i}2(a+b)\right)+z^{a+b}\exp\left(-\frac{\pi \i}2(a+b)\right)
\eeqq 
 lies in the sector $|\arg(w)|<1/16$. From here we find that
\beq\label{Delta_L2}
|\Delta \arg(A e^{kz}+ B z^{a+b}, L_2)|< \frac{1}{8},
\eeq
and at the same time, with the help of property (\ref{fact2}) we deduce that 
for all $z\in L_2$ 
\beq\label{estimate_L2}
\big| A e^{kz}+ B z^{a+b} \big| > \cos\left(\frac{1}{16} \right) 
\left( \big|A e^{kz} \big| +  \big |B z^{a+b} \big| \right).
\eeq   
Next, if  $g(z)= o\left(e^{kz}\right)+o\left(z^{a+b}\right)$ as $z\to \infty$, then 
\beqq
g(z)=o\left( \big|A e^{kz} \big| +  \big |B z^{a+b} \big| \right),
\eeqq
and we again can use (\ref{Delta_L2}) and (\ref{Delta_arg_property}) with $f(z):=F(z)$ and $g(z)$ defined above, to conclude  that for all $n$ large enough 
\beqq
|\Delta \arg(F(z),L_2(n))-\Delta \arg(A e^{kz}+ B z^{a+b}, L_2)|<1/8.
\eeqq 
The above inequality and estimate \eqref{Delta_L2} imply that for all $n$ large enough 
we have $|\Delta \arg(F(z),L_2(n))|<1/4$. Using exactly the same technique we obtain an identical estimate the change of argument over $L_4(n)$. 

Finally, on the contour $L_3$ we have $F(z)=A\exp(kz)+o(\exp(kz))$. Since $\Delta \arg(\exp(kz),L_3)=2\pi$, we use Proposition \ref{prop_Delta_arg_property}  and
 conclude that for all $n$ large enough $|\Delta(F,L_3(n))-2\pi|<1/4$. Combining these four estimates we see that
 for all $n$ large enough we have $|\Delta(F,L(n))-2\pi|<1$, and since we know that $\Delta \arg(F,L)$ must be an integer multiple of $2\pi$ we conclude that $\Delta \arg(F,L)=2\pi$, thus there is exactly one solution to $\psi(z)=q$ inside the contour $L(n)$. 
 From the first part of the proof we know that every sufficiently large solution to $\psi(z)=q$ must be close to $z_m$ for some $m$, and since by construction there is only one such point $z_n$ inside the contour $L(n)$, we conclude that for every $n$ large enough there is a solution to $\psi(z)=q$ close to $z_n$.
\qed
\\


{\it Proof of Proposition \ref{prop_asymptotic_psi}:}
Let us assume that $\alpha<1$ and $\hat \alpha<1$. Then we can take the cutoff function $h(x)\equiv 0$ in (\ref{Levy_Khinthine2}), and we can rewrite 
$\psi(z)$ as follows
\beq\label{psi_decomposition}
\psi(z)=\frac12 \sigma^2 z^2 +\mu z + \psi_1(z) + \psi_2(z),
\eeq 
where we have denoted
\beq\label{def_psi1}
\psi_1(z)&=&\int\limits_{(-\infty,0)} \left( e^{z x}-1\right) \Pi(\d x), \\ 
\psi_2(z)&=&\int\limits_{(0,k]} \left( e^{z x}-1\right) \Pi(\d x). \label{def_psi2}
\eeq

First, let us study the asymptotic behavior of $\psi_1(z)$. If $\hat \alpha< 0$ then part (1) of Definition 
\ref{definition_Levy_measure} implies that $\Pi((-\infty,0))<\infty$, thus 
$\psi_1(z)=O(1)$ as $z\to \infty$, $\re(z) \ge 0$.  If $0<\hat \alpha<1$, then part (1) of Definition 
\ref{definition_Levy_measure}  implies that for $x \in \r^-$
\beqq
\Pi(\d x)=\left( \hat C \hat \alpha |x|^{-1-\hat \alpha} + \sum\limits_{j=1}^{\hat m} \hat C_j \hat \alpha_j |x|^{-1-\hat \alpha_j} \right) \d x+ \nu(\d x),
\eeqq
where $\nu(\d x)$ is a finite measure on $\r^{-}$ (note that $\nu(\d x)$ does not have to be a positive measure). 
It is clear that 
\beqq
 \bigg | \int\limits_{-\infty}^0 \left( e^{zx}-1 \right) \nu(\d x) \bigg | < 2 |\nu ((-\infty, 0))|
\eeqq 
for $\re(z) \ge 0$. Using integration by parts we find that for $\re(z) > 0$
\beq\label{int_identity_1}
\hat \alpha \int\limits_{-\infty}^0 \left( e^{zx}-1 \right) |x|^{-1-\hat \alpha} \d x=-\Gamma(1-\hat \alpha) z^{\hat \alpha}.  
\eeq  
Combining the above three equations and (\ref{def_psi1}) we conclude that 
\beq\label{asymptotics_psi1}
\psi_1(z)=- \hat C \Gamma(1-\hat \alpha) z^{\hat \alpha} + o (z^{\hat \alpha})
\eeq
 as $z\to \infty$, $\re(z) > 0$.  

Next, let us investigate the asymptotic behavior of $\psi_2(z)$. Let us assume that $n>0$ (where $n$ is the constant in the Definition \ref{definition_Levy_measure}), 
the proof in the case $n=0$ is very similar. Since $n>0$, 
Definition \ref{definition_Levy_measure} implies that the measure $\Pi(\d x)$ restricted to $\r^+$ has a density $\pi(x)$, which belongs to 
${\mathcal{PC}}^{n}((0,k])$. Let us assume first that  $\pi \in {\mathcal{C}}^{n}((0,k])$, we will relax this assumption later.

First let us consider the case $\alpha < 0$, which is equivalent to $\Pi((0,\infty))<\infty$. 
Applying integration by parts $n$ times to (\ref{def_psi1}) we obtain 
\beqq
\psi_2(z)&=&\int\limits_0^k e^{zx} \pi(x) \d x - \int\limits_0^k \pi(x) \d x\\
&=&\sum\limits_{m=0}^{n-1} (-1)^m \left[ \pi^{(m)}(k-)e^{kz}z^{-m-1}-\pi^{(m)}(0+) z^{-m-1} \right]
+(-1)^n z^{-n} \int\limits_0^k \pi^{(n)}(x) e^{zx} \d x  - \int\limits_0^k \pi(x) \d x.
\eeqq
Since $\pi^{(n)}(x)$ is continuous we conclude that 
\beqq
\int\limits_0^k \pi^{(n)}(x) e^{zx} \d x= o \left( e^{kz} \right)
\eeqq
as $z\to \infty$, $\re(z) > 0$. At the same time, due to Definition \ref{definition_Levy_measure} we have $\pi^{(m)}(k-)=0$ for all $m\le n-2$. Using the above two results and 
the fact that $\frac{\d}{\d x} \Pi^+ (x)=-\pi(x)$ we obtain
\beq\label{asymptotics_psi2_1}
\psi_2(z)=(-1)^{n} \Pi^+ {}^{(n)}(k-) e^{kz}z^{-n}+o \left( e^{kz} z^{-n} \right) + O(1),
\eeq 
as $z\to \infty$, $\re(z) > 0$. Equation (\ref{asymptotics_psi2_1}) shows that the exponential term in the right-hand side of (\ref{psi_asymptotics})
comes from the upper boundary of the support of the L\'evy measure and from the first non-zero derivative of $\bar \Pi^+(x)$ at $k-$.

Next, let us assume that $\alpha \in (0,1)$. Then, according Definition \ref{definition_Levy_measure}, the density of the L\'evy measure can be expressed as follows
\beqq
\pi(x)=C\alpha x^{-1-\alpha}+\sum\limits_{j=1}^{m} C_j \alpha_j x^{-\alpha_j}+g(x),  
\eeqq
where $g \in {\mathcal{C}}^{n}([0,k])$.  We can rewrite $\psi_2(z)$ as 
\beq\label{expression_psi2}
\psi_2(z)= C F(\alpha,k,z)+ \sum\limits_{j=1}^{m} C_j F(\alpha_j,k,z)+ \int\limits_0^k e^{zx} g(x) \d x - \int\limits_0^k g(x) \d x,
\eeq
where we have defined 
\beqq
 F(\alpha,k,z)=\alpha \int\limits_0^k \left( e^{zx}-1 \right) x^{-1-\alpha} \d x. 
\eeqq
Let us obtain an asymptotic expansion of $F(\alpha,k,z)$ as $z\to \infty$, $z\in {\mathcal Q}_1$. Expanding $\exp(zx)$ in Taylor series centered at zero 
and integrating term by term we find that
\beq\label{eqn_F_alpha_k_z}
F(\alpha,k,z)=  k^{-\alpha} \left[ 1 - 
{}_1 F_1(-\alpha,1-\alpha;kz) \right],
\eeq
where ${}_1F_1(a,b;z)$ is the confluent hypergeometric function defined by (\ref{def_1F1}). 
Applying asymptotic formula (2) on page 278 in \cite{Erdelyi1955V3} we conclude that
\beq\label{asymptotics_F_alpha_k_z}
F(\alpha,k,z) &=&  -  \Gamma(1-\alpha) e^{-\pi \i \alpha} z^{\alpha}+
  \alpha k^{-1-\alpha} \frac{e^{kz}}{z} \sum\limits_{m=0}^N  \frac{(1+\alpha)_m}{(kz)^m} \\ \nonumber
 &+&O(1)+O\left(z^{\alpha-1}\right)
+O\left( e^{z}z^{-N-2}\right),
\eeq
as $z\to \infty$, $z\in {\mathcal Q}_1$. Formula (\ref{asymptotics_F_alpha_k_z}) and our previous result (\ref{asymptotics_psi2_1}) imply that 
\beq\label{psi_2_final}
\psi_2(z)&=& 
(-1)^{n} \Pi^+ {}^{(n)}(k-) e^{kz}z^{-n}-  \Gamma(1-\alpha) e^{-\pi \i \alpha} z^{\alpha}\\ \nonumber
&+&o \left(e^{kz}z^{-n} \right) + o\left(z^{\alpha}\right)+O(1),
\eeq
as $z\to \infty$, $z\in {\mathcal Q}_1$. 

As a final step, let us relax the assumption $\pi \in {\mathcal{C}}^{n}([0,k])$. Assume that there is a unique point $x_1 \in (0,k)$ at which
$\pi^{(n)}(x)$ does not exist (the proof in the general case is exactly the same). 
 According to Definition \ref{definition_Levy_measure}, $\pi \in {\mathcal C}^{n-2}(\r^+)$, thus
$\pi^{(m)}(x_1-)=\pi^{(m)}(x_1+)$ for $m\le n-2$. Applying integration by parts $n$ times on each subinterval $(0,x_1)$ and $(x_1,k)$ 
we would obtain an expression (\ref{psi_2_final}) plus an extra term of the form
\beqq
h(z)=(-1)^{n-1} \left[\pi^{(n-1)}(x_1-)-\pi^{(n-1)}(x_1+) \right] e^{x_1 z}z^{-n}.
\eeqq
However, it is easy to see that $h(z)=o(e^{kz}z^{-n}) + o(z^{\alpha})$ as $z\to \infty$, $z\in {\mathcal Q}_1$. 
This is true since in the domain ${\mathcal D}=\{\re(z)<\ln(\ln(\im(z))), \; \im(z)>e\}$ we have
$|\exp(x_1 z)|=\exp(x_1 \re(z))=O(\ln|z|)=o(z^{a})$, while in the domain ${\mathcal Q}_1\setminus {\mathcal D}$ we have $\re(z)\to \infty$ when 
$z\to \infty$, which implies $\exp(x_1 z)z^{-n}=o(\exp(kz)z^{-n})$.

Formula (\ref{psi_decomposition}) and asymptotic expressions (\ref{psi_2_final}), (\ref{asymptotics_psi1})  imply that 
$\psi(z)$ satisfies (\ref{psi_asymptotics}) with coefficients $A$, $a$, $B$ and $b$ as in Proposition \ref{prop_asymptotic_psi}, except that there
would be an extra term $O(1)$ in the right-hand side of (\ref{psi_asymptotics}), which comes from (\ref{psi_2_final}). According to our assumption, the process $X$ 
is not a compound Poisson process, thus the constant $b$ defined in Proposition \ref{prop_asymptotic_psi} is strictly positive, and $O(1)=o(z^b)$, therefore
this extra term can be absorbed into $o(z^b)$. This ends the proof in the case $\alpha<1$ and $\hat \alpha <1$. 

In the case when one or both of $\alpha$, $\hat \alpha$ are greater than one the proof is identical, except that we will have to do one extra integration by
parts for proving (\ref{asymptotics_psi1}). The details are left to the reader.
 \qed


{\it Proof of Theorem \ref{thm_w^q}:} 
Let us denote 
\beqq
z_n=\frac{1}{k}  \left[\ln \left( \bigg | \frac{B}{A} \bigg| \right)+(a+b) \ln \left(\frac{2n\pi }{k}\right) \right] +\frac{\i }{k} \left[ \arg\left( \frac{B}{A} \right)+\left( \frac12 (a+b) + 2n+1\right) \pi \right].
\eeqq
Due to Definition \ref{definition_Levy_measure}, the L\'evy measure $\Pi(\d x)$ can only have a finite number of atoms. From Corollary 2.5 in \cite{KuKyRi} we find that 
$W^{(q)}(x)$ can only have a finite number of points where it is not differentiable. Thus we can use (\ref{def_W^q}) and the Bromwich integral formula to conclude that 
for any $c> \Phi(q)$ 
\beq\label{proof_Wq_1}
W^{(q)}(x)=\frac{1}{2\pi \i } \int\limits_{c+\i \r} \frac{e^{zx}}{\psi_Y(z)-q} \d z.
\eeq

For $n>0$ and $m<0$ we define the contour $L=L(n,m)=L_1\cup L_2 \cup L_3 \cup L_4$, where  
\beqq
L_1&=&L_1(n)=\{ \re(z)=c, \; -\im(z_n)-\pi/k<\im(z)<\im(z_n)+\pi/k\},\\
L_2&=&L_2(n,m)=\{ \im(z)=\im(z_n)+\pi/k, \; m<\re(z)<c\},\\
L_3&=&L_3(n,m)=\{ \re(z)=m, \; -\im(z_n)-\pi/k<\im(z)<\im(z_n)+\pi/k\},\\
L_4&=&L_4(n,m)=\{ \im(z)=-\im(z_n)-\pi/k, \; m<\re(z)<c\}.
\eeqq
This contour is shown on figure \ref{fig_proof_b}. We assume that $L$ is oriented counter-clockwise. Using the residue theorem we deduce
\beq\label{sum_residues}
\int\limits_{L} \frac{e^{zx}}{\psi_Y(z)-q} \d z= \frac{e^{\Phi(q)x}}{\psi_Y'(\Phi(q))} + \frac{e^{-\zeta_0 x}}{\psi_Y'(-\zeta_0)}+ 2  \sum \re\left[ \frac{e^{-\zeta_j x}}{\psi_Y'(-\zeta_j)} \right],
\eeq
where the summation is over all $j\ge 1$, such that  $-\zeta_j$ lie inside the contour $L$. 

First, assume that $n$ is fixed and let us consider what happens as $m\to -\infty$. According to the asymptotic relation (\ref{psi_Y_asymptotics}), as $\re(z)\to -\infty$ the
function $\psi_Y(z)$ increases exponentially (uniformly in every horizontal strip $|\im(z)|<C$). In particular, for $m$ large enough we would have $|\psi_Y(z)-q|>1$ 
for all $z\in L_3(n,m)$, which implies
\beqq
\bigg | \int\limits_{L_3(n,m)} \frac{e^{zx}}{\psi_Y(z)-q} \d z \bigg| \le  \int\limits_{L_3(n,m)} \bigg | \frac{e^{zx}}{\psi_Y(z)-q} \bigg | \times |\d z|
< \int\limits_{L_3(n,m)} e^{x m}   \times |\d z| = \left(2\im(z_n)+2\pi /k\right)e^{mx},
\eeqq 
and for every $x>0$ the right-hand side converges to zero as $m\to -\infty$. 

Our next goal is to let $n\to +\infty$ and to prove that the integrals over the two horizontal half-lines 
$L_2(n,-\infty)$ and $L_4(n,-\infty)$ in \eqref{sum_residues} disappear. In order to achieve this we'll need to obtain good upper bounds on
$\psi(z)$ on these horizontal half-lines. Let us consider first the contour $L_2(n,-\infty)$. We will prove that there
exists a constant $C$ such that $|\psi_Y(z)|> C |\im(z_n)|$ for all $z\in L_2(n,-\infty)$. 

Assume that $\epsilon>0$ is a small number and define a domain 
\beqq
{\mathcal D}_{\epsilon}=\{z \in \c: \; |\arg(z)|>\pi/2+\epsilon\},
\eeqq
see figure  \ref{fig_proof_b}. 
Let $L_5=L_5(n)=L_2(n,-\infty) \cap {\mathcal D}_{\epsilon}$ and $L_6=L_6(n)=L_2(n,-\infty) \setminus {\mathcal D}_{\epsilon}$.  
 Following the same steps as in the proof of Theorem \ref{thm_asymptotics} (see estimate (\ref{estimate_L2})) we find that there exists a constant $C_1$ (which does not depend on $n$ or $\epsilon$) 
 such that for all $n$ large enough  we have for all $z\in L_5(n)$
\beq\label{proof_Wq_2}
|\psi_Y(z)|&>&\cos(C_1 \epsilon)\left( \big|Ae^{kz}z^{-a} \big| +  \big |Bz^{b} \big| \right)>\cos(C_1 \epsilon) |B|  |z|^b
\\ \nonumber &>&\cos(C_1 \epsilon) |B|\im(z_n)^b>\cos(C_1 \epsilon) |B|\im(z_n)
\eeq
where in the last estimate we have used the fact that $b\ge 1$ (see \eqref{eqn_AaBb}). 

Next, it can be easily seen from the figure \ref{fig_proof_b} that for all $z$ in the domain ${\mathcal D}_{\epsilon}$ we have 
\beq\label{estimate_for_De}
\re(z)<-|z| \sin(\epsilon),
\eeq
therefore  $|z|^b = o (\exp(-k z) z^{-a})$ when $z\to \infty$, $z\in {\mathcal D}_{\epsilon}$. This fact and 
the asymptotic formula (\ref{psi_Y_asymptotics}) show that there exists a constant $C_2$ such that for all $z \in {\mathcal D}_{\epsilon}$ large enough we have
$|\psi_Y(z)|>C_2 |\exp(-k z) z^{-a}|$. Therefore, for all $n$ large enough we have 
\beqq
|\psi_Y(z)|>C_2  e^{k|\re(z)|}|z|^{-a}, \;\;\; z\in L_6(n). 
\eeqq
Using the above estimate and \eqref{estimate_for_De} we find
\beq\label{proof_Wq_4}
|\psi_Y(z)|>C_2 \sin(\epsilon)^a e^{k|\re(z)|}|\re(z)|^{-a}, \;\;\; z\in L_6(n). 
\eeq
Next, as $n$ increases to $+\infty$, the real part of any $z \in L_6(n)$ decreases to $-\infty$ (see figure \ref{fig_proof_b}), thus for all $n$ 
large enough we have $\exp(k|\re(z)|/2)>|\re(z)|^a$ for all $z\in L_6(n)$. 
At the same time, from the figure \ref{fig_proof_b} we see that for all $z\in L_6(n)$ it is true that 
$|\re(z)|> \tan(\epsilon) |\im(z)|=\tan(\epsilon)(\im(z_n)+\pi/k)$.  
Using this fact and (\ref{proof_Wq_4}) we find that
there exists a constant $C_3=C_3(\epsilon)$ such that for all $n$ large enough
\beq\label{proof_Wq_5}
|\psi_Y(z)|>C_2 \sin(\epsilon)^a e^{\frac{k}2|\re(z)|}|\re(z)|^{-a} e^{\frac{k}2|\re(z)|}>
C_2 \sin(\epsilon)^a e^{\frac{k}2\tan(\epsilon) \im(z_n) }>C_3 \im(z_n), \;\;\; z\in L_6(n). 
\eeq
Combining (\ref{proof_Wq_2}) and (\ref{proof_Wq_5}) we conclude that there exists a constant $C>0$, such that for all $n$ large enough
we have
\beqq
|\psi_Y(z)|> C \im(z_n), \;\;\; z\in L_2(n,-\infty).
\eeqq
A similar estimate for $L_4(n,-\infty)$ can be obtained in the same way. 

Thus setting $z=z(u):=u+\i(\im(z_n)+\pi/k)$ we obtain
\beqq
&&\bigg | \int\limits_{L_2(n,-\infty)} \frac{e^{zx}}{\psi_Y(z)-q} \d z \bigg|=\bigg | \int\limits_{-\infty}^c \frac{e^{z(u)x}}{\psi_Y(z(u))-q} \d u \bigg|
\\&<&
\int\limits_{-\infty}^c \frac{|e^{z(u)x}|}{|\psi_Y(z(u))-q|} \d u<
\int\limits_{-\infty}^c \frac{e^{ux}}{ C |\im(z_n)|-q} \d u=
  \frac{x^{-1} e^{cx}}{ C |\im(z_n)|-q}
\eeqq
and the right hand side converges to zero as $n\to +\infty$. Similarly, the integral over $L_4$ vanishes. 
Thus as $n\to +\infty$ formula (\ref{sum_residues}) becomes
\beqq
\int\limits_{c+\i \r} \frac{e^{zx}}{\psi_Y(z)-q} \d z=\frac{e^{\Phi(q)x}}{\psi_Y'(\Phi(q))}  + \frac{e^{-\zeta_0 x}}{\psi_Y'(-\zeta_0)}
+ 2  \sum\limits_{n\ge 1} \re\left[ \frac{e^{-\zeta_n x}}{\psi_Y'(-\zeta_n)} \right]
\eeqq 
and the left-hand side is equal to $W^{(q)}(x)$ due to Bromwich integral formula (\ref{proof_Wq_1}).

Next, from Proposition \ref{prop_asymptotic_psi}  we know that the asymptotic formula for $\psi_Y'(z)$ can be obtained by differentiating
(\ref{psi_Y_asymptotics}). Therefore, using the asymptotic expression (\ref{zeta_asymptotics}) for $\zeta_n$ we find that 
\beqq
\big | \psi'(-\zeta_n) \big |=k|B| \left( \frac{2n\pi}{k} \right)^{b}+o\left(n^b\right).
\eeqq
Similarly, from (\ref{zeta_asymptotics}) we find that there exists a constant $c$ such that
\beqq
|e^{-\zeta_n x}|\sim c n^{-\frac{x}{k}(a+b)}, \;\;\; n\to +\infty,
\eeqq
thus the terms of the series in the right-hand side of (\ref{eqn_W^q}) decrease as $n^{-b-x(a+b)/k}$. 
According to (\ref{eqn_AaBb}) we always have $b\ge 1$, which implies that 
the series in the right-hand side of (\ref{eqn_W^q}) converges on $\r^+$ and uniformly on $[\epsilon,\infty)$ for each $\epsilon>0$. 
\qed


{\it Proof of Proposition \ref{Laplace_exponent_X}:} 
First, let us assume that $\alpha<1$ and $\hat \alpha<1$.
We start with the L\'evy-Khintchine formula (\ref{Levy_Khinthine2}) with the cutoff function $h(x)\equiv 0$, which gives us
 \beq\label{proof_Laplace_exp_X_1}
\psi(z)=\frac12 \sigma^2 z^2 +\mu z + \hat C \hat \alpha \int\limits_{-\infty}^{0} \left( e^{z x}-1 \right) e^{\hat \beta x} |x|^{-1-\hat \alpha} \d x
+ C  \alpha \int\limits_{0}^{k} \left( e^{z x}-1 \right)e^{- \beta x} x^{-1- \alpha} \d x.
\eeq
The first integral in (\ref{proof_Laplace_exp_X_1}) can be evaluated as follows:
\beqq
&&\hat C \hat \alpha \int\limits_{-\infty}^{0} \left( e^{z x}-1 \right) e^{\hat \beta x} |x|^{-1-\hat \alpha} \d x=
\hat C \hat \alpha \int\limits_{-\infty}^{0} \left( e^{(z+\hat \beta)x}-1 \right) |x|^{-1-\hat \alpha} \d x - 
\hat C \hat \alpha \int\limits_{-\infty}^{0} \left( e^{\hat \beta x}-1 \right) |x|^{-1-\hat \alpha} \d x \\ &=& 
- \hat C \Gamma(1-\hat\alpha) (\hat \beta+z)^{\hat \alpha}+ \hat C \Gamma(1-\hat\alpha) \hat \beta^{\hat \alpha},
\eeqq
where we have used (\ref{int_identity_1}) in the final step. Similarly, the second integral in (\ref{proof_Laplace_exp_X_1}) can be evaluated
with the help of (\ref{eqn_inc_gamma}) and (\ref{eqn_F_alpha_k_z}):
\beqq
C  \alpha \int\limits_{0}^{k} \left( e^{z x}-1 \right)e^{- \beta x} x^{-1- \alpha} \d x&=&
C  \alpha \int\limits_{0}^{k} \left( e^{(z-\beta) x}-1 \right) x^{-1- \alpha} \d x-
C  \alpha \int\limits_{0}^{k} \left( e^{-\beta x}-1 \right) x^{-1- \alpha} \d x \\
&=&C F(\alpha,k,z-\beta)-C F(\alpha,k,-\beta)\\&=&
-Ck^{-\alpha} {}_1F_1(-\alpha,1-\alpha,-k(\beta-z))+C k^{-\alpha}{}_1F_1(-\alpha,1-\alpha,-k\beta)\\&=&
C \alpha (\beta-z)^{\alpha} \gamma(-\alpha,k(\beta-z))-C \alpha \beta^{\alpha} \gamma(-\alpha,k\beta).
\eeqq

This ends the proof in the case $\alpha<1$ and $\hat \alpha<1$. When $\alpha>1$ or $\hat \alpha>1$ the proof would be very similar, the only difference 
is that we would use the cutoff function $h(x)\equiv 1$ 
in (\ref{Levy_Khinthine2}) and perform an extra 
integration by parts in (\ref{int_identity_1}). We leave all the details to the reader.
\qed



\end{document}